\documentclass[a4paper,11pt,reqno]{amsart}

\usepackage[utf8]{inputenc}
\usepackage[english]{babel}
\usepackage{amsmath}
\usepackage{amsfonts}
\usepackage{amsthm}
\usepackage{mathtools}
\usepackage{float}
\usepackage{placeins}
\usepackage{stmaryrd}
\usepackage{graphics}
\usepackage{graphicx}
\usepackage{subfig}
\usepackage{enumitem}
\usepackage{datetime}
\usepackage{bm}
\usepackage[colorlinks=false]{hyperref}
\usepackage{color}

\newcommand{\mc}{\mathcal}
\newcommand{\mf}{\mathfrak}
\newcommand{\eps}{\varepsilon}
\renewcommand{\d}{\,\mathrm{d}}
\newcommand{\tr}{\mathrm{tr}}
\renewcommand{\div}{\mathrm{div}\,}

\DeclareMathOperator*{\argmin}{argmin}

\def\N{\mathbb{N}}
\def\R{\mathbb{R}}
\def\C{\mathbb{C}}
\renewcommand{\to}{\rightarrow}

\numberwithin{equation}{section}
\newtheorem{thm}{Theorem}[section]
\newtheorem{defn}[thm]{Definition}
\newtheorem{prop}[thm]{Proposition}
\newtheorem{lemma}[thm]{Lemma}
\newtheorem{cor}[thm]{Corollary}

\theoremstyle{definition}
\begingroup
\newtheorem{rmk}[thm]{Remark}
\newtheorem{ex}[thm]{Example}
\endgroup

\theoremstyle{remark}

\begin{document}
	\author[F. Freddi and F. Riva]{Francesco Freddi and Filippo Riva}
	
	\title[Potential- and non potential-based cohesive models]{Potential-based  versus non potential-based cohesive models accounting for loading and unloading with application to sliding elastic laminates}
	
	\begin{abstract}
		A rigorous unified perspective of cohesive zone models is presented, including and comparing potential-based and non potential-based formulations, and encompassing known examples studied in literature. The main novelty of the work consists in the natural inclusion of loading and unloading effects in a general mixed-mode framework, incorporated through an intrinsic construction of energy densities or tensions. The proposed mathematical investigation identifies and proves the limitations of variational models with respect to non-variational ones, the latter yielding a feasible description of real instances in all relevant situations and regimes. This validates existing empirical and numerical observations. An application to a mechanical process of two elastic laminates sliding one on each other along their cohesive interface is finally analyzed, and existence results in both potential-based and non potential-based versions are obtained, extending previous contributions.
	\end{abstract}
	
	\maketitle
	
	{\small
		\par
	}

	\pagenumbering{arabic}
	
	\medskip
	
	\tableofcontents
	
	\section*{Introduction}
	
	Over the past two decades, variational mathematical methods have emerged as a powerful tool for studying failure phenomena in solids, driven by their intrinsic connection to the energetic nature of mechanical processes. These methods have gained significant interest not only in the mathematical community but also in engineering applications, particularly in fracture mechanics and related problems such as debonding and delamination. For comprehensive overviews of these approaches, we refer to \cite{BourFrancMar,FrancMar}. In this context, mechanical models that can be described by minimizing a suitable energy are often referred to as potential-based models in the engineering literature, while those characterized solely by equilibrium equations between strain and tension are termed non potential-based models. Mathematically, these are commonly distinguished as variational and non-variational models, respectively. Although potential-based models are frequently preferred in applications due to their simpler structure, they are not universally applicable and may fail to adequately capture certain phenomena, as highlighted in \cite{McGarry}.	
	
	Within fracture mechanics and delamination processes, models can be further categorized into cohesive and brittle frameworks. Cohesive models, pioneered by Barenblatt \cite{Barenblatt} and Dugdale \cite{Dugdale}, describe failure as a gradual process, in contrast to brittle models, which assume abrupt collapse. Brittle fracture, first energetically interpreted by Griffith \cite{Griffith}, occurs instantaneously when a critical threshold (often termed toughness) is exceeded. Cohesive models, on the other hand, are particularly effective in capturing the distinct regimes of loading and unloading, where a variable such as crack amplitude or displacement slip either increases or decreases. This is achieved through the introduction of an irreversible history variable, which tracks the system's past states and distinguishes between loading (when the variable increases) and unloading (when it remains constant). Similar approaches have been extended to model fatigue \cite{CrisLazzOrl}, where the history variable may also evolve during unloading phases.
	
	A key challenge in this field is the construction of cohesive energies and tension expressions that incorporate both current and history variables while maintaining physically feasible mechanical properties. While potential-based models involving only current variables--suitable for systems under monotone loading--are well-documented in the literature (see, e.g., \cite{ParkPaulino}), models that account for unloading regimes remain scarce and incomplete. For instance, in \cite{PPRthermo}, the history variable is treated as a damage parameter rather than being intrinsically embedded within the potential, limiting the model's applicability. Recent works such as \cite{BonCavFredRiva,Riv} have explored potentials incorporating unloading effects, but these studies are restricted to isotropic behaviors, where the energy and history variable depend solely on the amplitude of failure, not on its direction. For non potential-based models, the situation is even more limited: while expressions involving current variables have been proposed \cite{McGarry}, analytical frameworks for incorporating history variables into cohesive tensions remain largely unexplored.
	
	In this paper, we contribute to the analysis of cohesive-zone models in three significant ways. First, we provide a rigorous mathematical formulation of both variational and non-variational cohesive models under a unified framework. We propose an intrinsic method to construct potential energies and tension expressions that account for loading and unloading effects in a general anisotropic (mixed-mode) setting, starting from energy densities or tensions defined solely for the loading phase. These constructions yield the only admissible candidates with physically reasonable behavior, such as linear unloading responses following loading phases.
	
	Second, we compare potential-based and non potential-based models, demonstrating--through both theoretical analysis and representative examples--the limitations of the former relative to the latter. Specifically, we extend the observations of \cite{McGarry} to general loading-unloading scenarios, showing that the variational model is consistent only when the material exhibits the same fracture energy in all directions. Even in this case, the model predicts realistic behavior only under unidirectional unloading or in the restrictive case of uncoupled energy densities, where changes in one direction are independent of others. In contrast, the non-variational model produces feasible results across all loading-unloading regimes without such limitations.
	
	Finally, we investigate the consistency of these constructions in a specific model of a hybrid composite comprising two elastic laminates subjected to horizontal stretching driven by a prescribed boundary displacement. The laminates interact along their interface, where cohesive effects arise due to displacement mismatches. Unlike previous studies \cite{AleFredd2d,Riv}, which assume isotropic interfaces, we consider an anisotropic interface where cohesive effects depend on the direction of the displacement slip. While prior works focused solely on variational formulations, we establish well-posedness for the non-variational formulation as well. The model assumes slow, quasistatic evolution and small displacements, allowing for a linearized elasticity framework. This extends the scope of earlier studies on isotropic cohesive interfaces \cite{AleFredd1d,BalBabBouHenMau,BalBouMarMau, CagnToad, NegSca,NegVit} to the anisotropic case.
	
	\textbf{Plan of the paper.} In Section~\ref{sec:setting}, we provide a detailed description of the mechanical model under study: two elastic laminates with a cohesive interface. We present both the variational and the non-variational formulations of the problem. For the variational formulation, we introduce the elastic and cohesive energies of the system, while for the non-variational formulation, we derive the equilibrium equations (or inclusions) that must be satisfied. In both cases, we outline the key assumptions, distinguishing between those required for mathematical rigor and those essential from a mechanical perspective. We then define the notions of quasistatic evolution appropriate for each setting: for the potential-based model, we adopt the well-established concept of energetic solutions, while for the non potential-based model, we introduce a natural notion of equilibrium solutions.
	Section~\ref{sec:examples} focuses on the explicit construction of an anisotropic cohesive energy and a cohesive tension that incorporate both loading and unloading regimes, starting from a given energy density or tension defined solely for the loading phase. We ensure that these constructions satisfy all necessary mechanical and mathematical properties. Additionally, we highlight the limitations of the variational formulation compared to the non-variational one, emphasizing the latter's broader applicability.
	In Section~\ref{sec:numeric}, we present representative examples to validate the consistency of our theoretical constructions and to support the discussions in Section~\ref{sec:examples}. These examples illustrate the practical implications of our findings and demonstrate the effectiveness of the proposed models.
	Finally, Section~\ref{sec:existence} is dedicated to the mathematical proof of existence for both energetic solutions (in the variational setting) and equilibrium solutions (in the non-variational setting). For the former, we employ the well-known method of minimizing movements, while for the latter, we utilize Kakutani's fixed point theorem for set-valued functions. These proofs establish the well-posedness of the proposed formulations and provide a rigorous foundation for their application.

	\section*{Notation}
	
	The maximum (resp. minimum) of two extended real numbers $\alpha,\beta\in \R\cup\{\pm\infty\}$ is denoted by $\alpha\vee\beta$ (resp. $\alpha\wedge\beta$).
	
	For a positive integer $n\in \N$, the standard scalar product between vectors $v,w\in \R^n$ is denoted by $v\cdot w$ and we write $|v|_n$ for the euclidean norm, where the subscript simply stresses the space dimension. Analogously, by $0_n$ we mean the null vector in $\R^n$. For lightness of notation, in the scalar case $n=1$ we omit the subscript. We also introduce the vector $v\bm{\vee}w$ whose components are obtained by taking the maximum between the corresponding components of $v$ and $w$, i.e. $(v\bm{\vee}w)_i:=v_i\vee w_i$.
	
	We use the symbols $\R^{n\times n}$ and $\R^{n\times n}_{\rm sym}$ to denote the set of real $(n\times n)$-matrices and the subset of symmetric matrices. For any matrix $A\in \R^{n\times n}$, we write $A_{\rm sym}:=\frac12 (A+A^T)\in \R^{n\times n}_{\rm sym}$ for its symmetric part. In the case $A=\nabla u$ we adopt the usual notation $e(u)$ in place of $(\nabla u)_{\rm sym}$. The Frobenius scalar product between two matrices $A, B \in \R^{n\times n}$ is $A:B=\tr(AB^T)$, and the corresponding norm is denoted by $|A|_{n\times n}:=\sqrt{A:A}$.  The tensor product between two vectors $v,w\in \R^n$ is the matrix $v\otimes w\in \R^{n\times n}$ defined by $(v\otimes w)_{i,j}:=v_iw_j$.
	
	We adopt standard notations for Bochner spaces and for scalar- or vector-valued Lebesgue and Sobolev spaces. By $L^0(\Omega)^+$ we mean the space of nonnegative Lebesgue measurable functions on the (open) set $\Omega\subseteq \R^n$. Given $\alpha\in (0,1]$, by $C^{0,\alpha}(\overline{\Omega})$ and $C^{0,\alpha}(\overline{\Omega};\R^m)$ we mean, respectively, the space of scalar- and $\R^m$-valued functions which are $\alpha$-H\"older continuous (Lipschitz continuous if $\alpha=1$) in $\overline\Omega$, endowed with the norm $\|\cdot\|_{C^{0,\alpha}(\overline{\Omega})}:=\|\cdot\|_{C^0(\overline{\Omega})}+[\,\cdot\,]_{\alpha,\overline\Omega}$, where $[\,f\,]_{\alpha,\overline\Omega}:=\sup\limits_{\substack{{x,y\in\overline\Omega}\\x\neq y}}\frac{|f(x)-f(y)|}{|x-y|_n^{\alpha}}$ is the H\"older seminorm of $f$. In order to lighten the notation, we write the same symbol for the norms in $C^{0,\alpha}(\overline{\Omega})$ and in $C^{0,\alpha}(\overline{\Omega};\R^m)$; the meaning will be clear from the context. The same convention is used for norms in Lebesgue or Sobolev spaces. We finally denote by $C^{0,\alpha}_{\rm loc}({\Omega})$ (resp. $C^{0,\alpha}_{\rm loc}({\Omega};\R^m)$) the space of functions belonging to $C^{0,\alpha}(\overline{\Omega'})$ (resp. $C^{0,\alpha}(\overline{\Omega'};\R^m)$) for all open set $\Omega'\subset\subset \Omega$, i.e. such that the closure of $\Omega'$ is still a subset of $\Omega$.
	
	Positive and negative part of a real function $f$ are denoted by $f^+:=f\vee 0$ and $f^-:=-(f\wedge 0)$, respectively.
	
	Given a normed space $(X,\|\cdot\|_X)$, by $B([a,b];X)$ we mean the space of everywhere defined measurable functions $f\colon [a,b]\to X$ which are bounded in $X$, namely $\sup\limits_{t\in [a,b]}\|f(t)\|_X<+\infty$. 	
	
	\section{Setting}\label{sec:setting}
	
	We first describe the specific mechanical model of sliding laminates we intend to investigate in this paper, showing how  cohesive effects comes into play. We present both the variational and the non-variational formulation of the problem, and we introduce the corresponding notions of solution. We then state our main mathematical results, ensuring existence of such solutions for the two variants of the model.
	
	The reference configuration of the elastic composite is represented by an open, bounded, connected set $\Omega\subseteq\R^d$ with Lipschitz boundary. We point out that the physical dimension of the problem is $d=2$, but the mathematical arguments still remain rigorous in arbitrary dimension $d\in \N$.
	
	\subsection{Potential-based model}
	The total elastic energy of the composite is given by the functional $\mc E\colon H^1(\Omega;\R^d)^2\to [0,+\infty)$ defined as
	\begin{equation*}
		\mc E(\bm{u}):=\sum_{i=1}^{2}\frac 12 \int_{\Omega} \C_i(x) e(u_i(x)):e(u_i(x))\d x.
	\end{equation*}
	Here, the bold letter $\bm u=(u_1,u_2)$ denotes the pair of (horizontal) displacements of the two layers, which completely describes the elastic behaviour of the material since we are working in a linearized setting (small deformations). Moreover $\C_i\colon \Omega\to\R^{d\times d\times d\times d}$ is the stiffness tensor of the $i$th layer; for $i=1,2$ we assume that
	\begin{enumerate}[label=\textup{(C\arabic*)}, start=1]
		\item \label{hyp:C1} $\C_i$ is uniformly continuous with modulus of continuity $\omega_i$,
	\end{enumerate}
	together with the usual assumptions in linearized elasticity
	\begin{enumerate}[label=\textup{(C\arabic*)}, start=2]	
		\item \label{hyp:C2} $\C_i(x)A\in \R^{d\times d}_{\rm sym}$ for all $x\in\Omega$ and $A\in \R^{d\times d}$;
		\item \label{hyp:C3} $\C_i(x)A=\C_i(x)A_{\rm sym}$ for all $x\in\Omega$ and $A\in \R^{d\times d}$;
		\item \label{hyp:C4} $\C_i(x)A:B=\C_i(x)B:A$ for all $x\in\Omega$ and $A,B\in \R^{d\times d}$ (symmetry);
		\item \label{hyp:C5} $\C_i(x)A:A\ge c_i \vert A_{\rm sym}\vert^2_{d\times d}$ for some $c_i>0$ and for all $x\in\Omega$ and $A\in \R^{d\times d}$ (coercivity).
	\end{enumerate}
	We recall that the coercivity condition \ref{hyp:C5} automatically implies the so-called strict Legendre-Hadamard condition (see for instance \cite[end of Chapter 5]{Ciarlet})
	\begin{equation}\label{eq:LegendreHadamard}
		\C_i(x)(v\otimes w):(v\otimes w)\ge \frac{c_i}{2} \vert v\otimes w\vert^2_{d\times d},\qquad \text{ for all $x\in\Omega$ and $v,w\in \R^{d}$}.
	\end{equation}
	
	The (planar) interface between the two layers of material is assumed to behave in a cohesive fashion with respect to their reciprocal slip. Differently from previous contributions \cite{BonCavFredRiva, Riv}, we allow for anisotropy of such interface, possibly due to asymmetries in the microstructures of the strata. To describe such anisotropic effects we fix an integer $m\in \N$, representing the number of cohesive variables, and we consider a function $\mf g\colon \R^d\to [0,+\infty)^m$ satisfying
	\begin{enumerate}[label=\textup{($\mf g$\arabic*)}, start=1]	
		\item \label{hyp:g1} $\mf g(0_d)=0_m$;
		\item \label{hyp:g2} $\mf g$ is Lipschitz continuous in $\R^d$.
	\end{enumerate}
	
	\begin{rmk}\label{ex:1}
		The isotropic case analyzed in \cite{BonCavFredRiva, Riv} can be recovered by choosing $m=1$ and $\mf g(\delta)=|\delta|_d$, for $\delta\in \R^d$, namely the cohesive energy \eqref{eq:cohesiveenergy} below just depends on the amplitude of the slip between the two laminates, but not on its direction.
	\end{rmk}
	
	\begin{ex}\label{ex:2}
		The prototypical anisotropic (also called mixed-mode) case is given by the choice $m=d$ and $\mf g(\delta)=(|\delta_1|, |\delta_2|,\dots,|\delta_d|)$, namely the cohesive energy reacts differently with respect to slips in different directions. One may also replace the modulus $|\delta_i|$ with the asymmetric version $\mu_i^+\delta_i^++\mu_i^-\delta_i^-$, with coefficients $\mu^\pm_i\ge 0$, thus distinguishing between forward and backward slips.
	\end{ex}
	
	The cohesive energy is then given by a functional $\mc K\colon L^0(\Omega;\R^d)\times (L^0(\Omega)^+)^m\to[0,+\infty)$ defined by
	\begin{equation}\label{eq:cohesiveenergy}
		\mc K(\delta,\gamma):=\int_{\Omega}\Phi(x,\mf g(\delta(x)),\gamma(x))\d x,
	\end{equation}
	where the variable $\delta$ denotes the current slip of the two layers $u_1-u_2$, while $\gamma$, here and henceforth called history variable, represents the irreversible counterpart of the cohesive variables $\mathfrak g(\delta)$, namely each component $\gamma_l$ (at time $t$) plays the role of the maximum value reached by $\mathfrak g_l(\delta)$ during the evolution (till time $t$).
	
	The density $\Phi\colon \Omega\times [0,+\infty)^m\times [0,+\infty)^m\to [0,+\infty)$, which takes into account different loading-unloading regimes, is measurable and satisfies the following properties:
	\begin{enumerate}[label=\textup{($\Phi$\arabic*)}, start=1]	
		\item \label{hyp:phi1} $\Phi(x,0_m,0_m)=0$ for a.e. $x\in \Omega$;
		\item \label{hyp:phi2} $\Phi(x,\cdot,\cdot)$ is bounded, essentially uniformly with respect to $x\in \Omega$; also, it is continuous in the whole $[0,+\infty)^m\times [0,+\infty)^m$, for a.e. $x\in \Omega$;
		\item \label{hyp:phi3} the map $y\mapsto \Phi(x,y,z)$ is Lipschitz continuous in $[0,+\infty)^m$, essentially uniformly with respect to $x\in \Omega$ and uniformly with respect to $z\in[0,+\infty)^m$;
		\item \label{hyp:phi4} $\Phi(x,y,z)=\Phi(x,y,y\bm{\vee}z)$ for a.e. $x\in\Omega$ and for all $(y,z)\in [0,+\infty)^m\times [0,+\infty)^m$;
		\item \label{hyp:phi5} for a.e. $x\in\Omega$ and for all $y\in [0,+\infty)^m$ the map $z\mapsto\Phi(x,y,z)$ is nondecreasing with respect to each component.
	\end{enumerate}
	In Section~\ref{sec:examples} we propose an intrinsic way to explicitely construct densities $\Phi$ which behave in a proper fashion with respect to loading ($\mathfrak g_l(\delta)$ increases) and unloading ($\mathfrak g_l(\delta)$ decreases). We also present some examples, which encompass known models analyzed in literature \cite{AleFredd2d,ParkPaulino,PPRthermo}. In particular, we stress that realistic densities should fulfil the following properties (for a.e. $x\in\Omega$ and for all $z\in [0,+\infty)^2$):
	
	\begin{enumerate}[label=\textup{($\Phi$\arabic*)}, start=6]
		\item \label{hyp:phi6} for all $l=1,\dots,m$, the map $y_l\mapsto \Phi(x,y,z)$ is nondecreasing; 
		\item \label{hyp:phi7} for all $l=1,\dots,m$, the map $y_l\mapsto \Phi(x,y,z)$ is convex and quadratic in the unloading zone $y_l<z_l$, while it is concave in the loading zone $y_l\ge z_l$;
		\item \label{hyp:phi8} for all $l,j=1,\dots,m$ with $l\neq j$, the map $y_j\mapsto\partial_{y_l}\Phi(x,y,z)$ is nonincreasing;
		\item \label{hyp:phi9}$\lim\limits_{|y|_m\to \infty}|\nabla_y\Phi(x,y,z)|_m=0$.
	\end{enumerate}
	
	\begin{rmk}\label{rmk:initelast}
		Condition \ref{hyp:phi7} may be weakened requiring that the density is concave in the loading zone just if $z_l$ has reached a certain positive but small threshold $\bar{z}_l$. This describes materials possessing an initial elastic behaviour, included in the so-called intrinsic cohesive models.
	\end{rmk}	
	
	The evolution of the system is driven by a prescribed horizontal external displacement $\ell$ acting on a portion of a boundary $\partial_D\Omega\subseteq\partial\Omega$ of positive Hausdorff measure $\mc H^{d-1}(\partial_D\Omega)>0$. Given $T>0$, we require that
	\begin{equation}\label{eq:externalloading}
		\ell\in W^{1,1}(0,T;H^1(\Omega;\R^d)),
	\end{equation}
	and we postulate that it is slow with respect to internal vibrations, so that inertia can be neglected and the model can be set in a quasistatic framework.
	
	Given a function $f\colon\partial_D\Omega\to \R^d$, we also introduce the following notation:
	\begin{equation*}
		H^1_{D,f}(\Omega;\R^d):=\{v\in H^1(\Omega;\R^d):\, v=f\quad\mc H^{d-1}\text{-a.e. in }\partial_D\Omega\}.
	\end{equation*}
	The total energy of the elastic composite can thus be written through the functional $\mc F\colon [0,T]\times H^1(\Omega;\R^d)^2\times (L^0(\Omega)^+)^m\to [0,+\infty]$ defined by
	\begin{equation*}
		\mc F(t,\bm u, \gamma):=\begin{cases}
			\mc E(\bm u)+\mc K(u_1-u_2,\gamma),&\text{if }\bm u\in (H^1_{D,\ell(t)}(\Omega;\R^d))^2,\\
			+\infty,&\text{otherwise.}
		\end{cases}
	\end{equation*}
	
	At the initial time $t=0$ we prescribe the initial conditions
	\begin{subequations}\label{eq:initialdata}
		\begin{equation}\label{eq:in1}
			(\bm u^0,\gamma^0)\in (H^1_{D,\ell(0)}(\Omega;\R^d))^2\times C^{0,1}_{\rm loc}(\Omega;\R^m),
		\end{equation}
		and we assume that
		\begin{equation}\label{eq:initialstability}
			\begin{gathered}
				\gamma^0_l\ge \mf g_l(u^0_1-u_2^0),\quad\text{ for all }l=1,\dots m, \quad\text{ and }\\
				\mc F(0,\bm u^0,\gamma^0)\le \mc F(0,\bm v,\gamma^0),\quad\text{ for every }\bm v\in H^1(\Omega;\R^d)^2.
			\end{gathered}
		\end{equation}
	\end{subequations}
	
	The notion of solution we adopt in this paper for the potential-based model has a natural variational flavour, and it is well-fitted for quasistatic evolutions \cite{MielkRoubbook}. Roughly speaking, such solution minimizes at all times the total energy $\mc F(t,\cdot,\cdot)$ (see \ref{GS}) while the history variable increases, and at the same time an energy balance is preserved (see \ref{EB}).
	
	\begin{defn}\label{def:genensol}
		Given a prescribed displacement $\ell$ and initial data $(\bm u^0,\gamma^0)$ satisfying \eqref{eq:externalloading} and \eqref{eq:initialdata}, we say that a map  $[0,T]\ni t\mapsto (\bm u(t),\gamma(t))\in  H^1(\Omega;\R^d)^2\times (L^0(\Omega)^+)^m$ is a (generalized) \emph{energetic solution} to the potential-based cohesive interface model if the initial conditions $(\bm u(0),\gamma(0))=(\bm u^0,\gamma^0)$ are attained, each component of the history variable $\gamma$ is nondecreasing in time, and the following global stability condition and energy balance are satisfied for all $t\in [0,T]$:
		\begin{enumerate}[label=\textup{(GS)}]
			\item \label{GS} $\gamma_l(t)\ge \mf g_l(u_1(t)-u_2(t)),\quad\text{ for all }l=1,\dots m,\quad\text{ and }$\\$\mc F(t,\bm u(t),\gamma(t))\le \mc F(t,\bm v,\gamma(t)),\quad \text{ for every }\bm v\in H^1(\Omega;\R^d)^2;$
		\end{enumerate}
		\begin{enumerate}[label=\textup{(EB)}]
			\item \label{EB} $\displaystyle \mc F(t,\bm u(t),\gamma(t))=\mc F(0,\bm u^0,\gamma^0)+\mc W(t);$
		\end{enumerate}
		where the quantity $\mc W(t)$ represents the  amount of work computed by the prescribed displacement until the time $t$, which is defined as
		\begin{equation}\label{eq:work}
			\mc W(t):=\int_{0}^{t}\sum_{i=1}^2\int_{\Omega}\C_i(x) e(u_i(s,x)):e(\dot{\ell}(s,x))\d x\d s.
		\end{equation}
	\end{defn}
	
	Observe that the minimality requirement in \ref{GS} implies that an energetic solution formally solves the following system of partial differential inclusions at all times $t\in [0,T]$:
	\begin{equation}\label{eq:EL}
		\begin{cases}
			-\operatorname{div}\mathbb C_1 e(u_1(t))\in-\nabla_y \Phi(\cdot, \mathfrak g(u_1(t)-u_2(t)),\gamma(t))D\mathfrak g(u_1(t)-u_2(t)), &\text{in }\Omega,\\
			-\operatorname{div}\mathbb C_2 e(u_2(t))\in\nabla_y\Phi(\cdot, \mathfrak g(u_1(t)-u_2(t)),\gamma(t))D\mathfrak g(u_1(t)-u_2(t)), &\text{in }\Omega,\\
			u_1(t)=u_2(t)=\ell(t), &\text{in }\partial_D\Omega,\\
			\partial_n u_1(t)=\partial_n u_2(t)=0_d, &\text{in }\partial\Omega\setminus\partial_D\Omega,
		\end{cases}	
	\end{equation}
	where $D\mathfrak g(\delta)$ denotes the set of matrices $m\times d$ whose $l$-th row belongs to the (convex) subdifferential of $\mathfrak g_l$ at $\delta$.
	
	We now state our first mathematical result, which provides existence of energetic solutions. Its proof can be found in Section~\ref{sec:existence}.
	
	\begin{thm}\label{thm:existence}
		Let the stiffness tensors $\C_i$ satisfy \ref{hyp:C1}-\ref{hyp:C5} and let the cohesive energy density $\Phi$ and the cohesive variables $\mf g$ satisfy \ref{hyp:phi1}-\ref{hyp:phi5} and \ref{hyp:g1}-\ref{hyp:g2}. Then, given a prescribed displacement $\ell$ and initial data $(\bm u^0,\gamma^0)$ fulfilling \eqref{eq:externalloading} and \eqref{eq:initialdata}, there exists an energetic solution $(\bm u,\gamma)$ to the potential-based cohesive interface model in the sense of Definition~\ref{def:genensol}.
		
		Furthermore, such pair of displacements $\bm u$ actually belongs to $B([0,T]; H^1(\Omega;\R^d)^2)$ and to $B([0,T];C^{0,\alpha}(\overline{\Omega'};\R^d)^2)$ for all $\Omega'\subset\subset \Omega$ and $\alpha\in (0,1)$, while the history variable $\gamma$ actually is in $B([0,T];C^{0,\alpha}(\overline{\Omega'};\R^m))$ for all $\Omega'\subset\subset \Omega$ and $\alpha\in (0,1)$. 
	\end{thm}
	
	\subsection{Non potential-based model}
	As explained in the Introduction, non-variational models are described through equilibrium equations rather than energies. In view of the Euler-Lagrange equations \eqref{eq:EL}, the non potential version of the model under consideration is thus characterized by the following system:
	\begin{equation}\label{eq:system}
		\begin{cases}
			-\operatorname{div}\mathbb C_1 e(u_1(t))\in-\mathcal T(\cdot, \mathfrak g(u_1(t)-u_2(t)),\gamma(t))D\mathfrak g(u_1(t)-u_2(t)), &\text{in }\Omega,\\
			-\operatorname{div}\mathbb C_2 e(u_2(t))\in\mathcal T(\cdot, \mathfrak g(u_1(t)-u_2(t)),\gamma(t))D\mathfrak g(u_1(t)-u_2(t)), &\text{in }\Omega,\\
			u_1(t)=u_2(t)=\ell(t), &\text{in }\partial_D\Omega,\\
			\partial_n u_1(t)=\partial_n u_2(t)=0_d, &\text{in }\partial\Omega\setminus\partial_D\Omega,
		\end{cases}	
	\end{equation}
	where the right-hand side $\mp\mathcal T(\cdot, \mathfrak g(u_1(t)-u_2(t)),\gamma(t))D\mathfrak g(u_1(t)-u_2(t))\in \R^d$ represents the tension acting on the two layers. With a slight abuse of terminology, we still refer to the vector field  $\mathcal T\colon \Omega\times [0,+\infty)^m\times[0,+\infty)^m\to \R^m$ with the name cohesive tension. We also stress that it is not necessarily a gradient, as in the potential-based case. We require that $\mc T$ satisfies the following assumptions:
	\begin{enumerate}[label=\textup{($\mathcal T$\arabic*)}, start=1]	
		\item \label{hyp:T1} $\mathcal T(x,\cdot,\cdot)$ is bounded, essentially uniformly with respect to $x\in \Omega$; also, it is continuous in the whole $[0,+\infty)^m\times [0,+\infty)^m$, for a.e. $x\in \Omega$;
		\item \label{hyp:T2} $\mathcal T(x,y,z)=\mathcal T(x,y,y\bm{\vee}z)$ for a.e. $x\in\Omega$ and for all $(y,z)\in [0,+\infty)^m\times [0,+\infty)^m$.
	\end{enumerate}
	Analogously to the previous formulation, Section~\ref{sec:examples} also contains an intrinsic construction of tensions $\mathcal T$ possessing a realistic behaviour in both loading and unloading regimes. In particular, they also fulfil (for a.e. $x\in\Omega$ and for all $z\in [0,+\infty)^2$):
	\begin{enumerate}[label=\textup{($\mathcal T$\arabic*)}, start=3]
		\item \label{hyp:T3} for all $l=1,\dots,m$, the component $\mathcal T_l(x,y,z)$ is nonnegative;
		\item \label{hyp:T4} for all $l=1,\dots,m$, the map $y_l\mapsto\mathcal T_l(x,y,z)$ is linear and nondecreasing in the unloading zone $y_l<z_l$, while it is nonincreasing in the loading zone $y_l\ge z_l$ (possibly if $z_l\ge \bar z_l>0$, see Remark~\ref{rmk:initelast});
		\item \label{hyp:T5} for all $l,j=1,\dots,m$ with $l\neq j$, the map $y_j\mapsto\mathcal T_l(x,y,z)$ is nonincreasing;
		\item \label{hyp:T6}$\lim\limits_{|y|_m\to \infty}|\mc T(x,y,z)|_m=0$.
	\end{enumerate}	
	
	Here, in addition to \ref{hyp:g1} and \ref{hyp:g2}, we also need to require that the cohesive function $\mathfrak g$ fulfils: 
	\begin{enumerate}[label=\textup{($\mf g$\arabic*)}, start=3]	
		\item \label{hyp:g3} each component of $\mathfrak g$ is convex;
		\item \label{hyp:g4} the following closure property holds: if $\delta_n\to\delta$ strongly in $L^2(\Omega;\R^d)$, $\eta_n\rightharpoonup \eta$ weakly in $L^2(\Omega;\R^{m\times d})$, and $\eta_n(x)\in D\mathfrak g(\delta_n(x))$ for a.e. $x\in\Omega$, then $\eta(x)\in D\mathfrak g(\delta(x))$ for a.e. $x\in\Omega$ as well.
	\end{enumerate}
	
	\begin{rmk}
		It is standard to check that the examples introduced in Remarks~\ref{ex:1} and \ref{ex:2} satisfy the above assumptions \ref{hyp:g3} and \ref{hyp:g4}.
	\end{rmk}
	
	For the non potential-based model, we consider the following notion of quasistatic solution, which somehow replaces the stronger global stability condition \ref{GS} with an equilibrium condition.
	
	\begin{defn}\label{def:geneqsol}
		Given an external displacement $\ell$ and initial data $(\bm u^0,\gamma^0)$ satisfying \eqref{eq:externalloading}, \eqref{eq:in1}, the first condition in \eqref{eq:initialstability} and solving \eqref{eq:system} for $t=0$ (see \eqref{eq:weaksol} below), we say that a map  $[0,T]\ni t\mapsto (\bm u(t),\gamma(t))\in  H^1(\Omega;\R^d)^2\times (L^0(\Omega)^+)^m$ is a (generalized) \emph{equilibrium solution} to the non potential-based cohesive interface model if the initial conditions $(\bm u(0),\gamma(0))=(\bm u^0,\gamma^0)$ are attained, each component of the history variable $\gamma$ is nondecreasing in time, and the following equilibrium condition is satisfied for all $t\in [0,T]$:
		\begin{enumerate}[label=\textup{(EQ)}]
			\item \label{EQ} $\gamma_l(t)\ge \mf g_l(u_1(t)-u_2(t)),\quad\text{ for all }l=1,\dots m,\quad\text{ and }$\\ $\bm u(t)\in (H^1_{D,\ell(t)}(\Omega;\R^d))^2$ is a weak solution to \eqref{eq:system}, namely there exists a function $\bm \eta(t)\in(L^\infty(\Omega;\R^{m\times d}))^2$ such that 
			\begin{equation*}
				\eta_i(t,x)\in D\mathfrak g(u_1(t,x)-u_2(t,x)), \qquad\text{ for a.e. }x\in\Omega,\text{ and for }i=1,2,
			\end{equation*}
			and for all $\bm\varphi \in (H^1_{D,0_d}(\Omega;\R^d))^2$ there holds
			\begin{equation}\label{eq:weaksol}
				\sum_{i=1}^2\int_\Omega \mathbb C_ie(u_i(t)):e(\varphi_i) \d x=-\int_\Omega\mathcal T(x,\mathfrak g(u_1(t)-u_2(t)),\gamma(t))\cdot \Big(\eta_1(t)\varphi_1-\eta_2(t)\varphi_2\Big) \d x.
			\end{equation}
		\end{enumerate}
	\end{defn}
	
	\begin{rmk}
		If the displacement pair $\bm u$ of an equilibrium solution is more regular in time, say $\bm u\in AC([0,T];H^1(\Omega;\R^d)^2)$, then the following energy balance is also satisfied for all $t\in [0,T]$:
		\begin{align*}
			&\mc E(\bm u(t))+\int_0^t\int_\Omega \mc T(x,\mathfrak g(u_1(\tau)-u_2(\tau)),\gamma(\tau))\cdot \Big(\eta_1(\tau)(\dot u_1(\tau){-}\dot\ell(\tau)){-}\eta_2(\tau)(\dot u_2(\tau){-}\dot\ell(\tau))\Big)\d x\d \tau\\
			=&	\mc E(\bm u^0)+\mc W(t),
		\end{align*}
		where the work of the external displacement $\mc W$ has been introduced in \eqref{eq:work}. In the particular case $\mathcal T=\nabla_y\Phi$, namely if the model admits a potential, then one formally recovers \ref{EB}.
		
		A possible approach leading to the validity of a suitable energy balance without additional time-regularity (which is not expected in general) may be the vanishing viscosity argument. However, this goes beyond the scopes of the present paper.
	\end{rmk}
	
	In this setting, we are able to show existence of equilibrium solutions. The proof of the following result is developed in Section~\ref{sec:existence}.
	
	\begin{thm}\label{thm:existencenonpot}
		Let the stiffness tensors $\C_i$ satisfy \ref{hyp:C1}-\ref{hyp:C5} and let the cohesive tension $\mc T$ and the cohesive variables $\mf g$ satisfy \ref{hyp:T1}-\ref{hyp:T2} and \ref{hyp:g1}-\ref{hyp:g4}. Then, given a prescribed displacement $\ell$ and initial data $(\bm u^0,\gamma^0)$ fulfilling \eqref{eq:externalloading}, \eqref{eq:in1}, the first condition in \eqref{eq:initialstability} and \eqref{eq:weaksol} for $t=0$, there exists an equilibrium solution $(\bm u,\gamma)$ to the non potential-based cohesive interface model in the sense of Definition~\ref{def:geneqsol}.
		
		Furthermore, the same regularity properties stated in Theorem~\ref{thm:existence} hold. 
	\end{thm}

	\section{Construction of cohesive energy densities and tensions}\label{sec:examples}
	In this section we present an intrinsic way of building a feasible loading-unloading energy density $\Phi$ and tension $\mc T$, starting from a ``purely loading'' density $\Psi$ or tension $\mc S$, respectively. Actually, as we will see, a good candidate for $\mc S$ is given by $\nabla\Psi\bm\vee 0_d$, indeed the issues observed in \cite{McGarry,PPR} for potential-based models occur when the partial derivatives $\partial_i\Psi$ become negative. We then provide several examples of densities $\Psi$, including many expressions studied in \cite{McGarry,ParkPaulino,PPR,PPRthermo}. 
	
	Although the model depicted in the previous section deals with cohesive delamination, we point out that the arguments of the current section are valid for general cohesive models, so for instance they may be applied to fracture mechanics.
	
	Also supported by the numerical examples performed in Section \ref{sec:numeric}, our analysis essentially proves that variational mixed-mode models are consistent just if the complete delamination energies in the different directions are equal (see \eqref{eq:phi12}). Moreover, even in this case, they return physically realistic responses only in the restrictive (and quite unnatural) situation of unloading happening in one direction, while the slips in the other ones are still. On the contrary, our construction of non-variational mixed-mode models provides a satisfactory description of real instances without any limitations and in any loading-unloading regime.
	
	For applicative reasons, we focus on the physical dimension $d=2$; moreover, for the sake of clarity we restrict our attention to homogeneous interfaces, i.e. the densities do not depend on $x\in\Omega$, but we stress that inhomogeneous ones can be treated in an analogous way.
	
	\subsection{Construction of the potential $\Phi$} We consider a loading function $\Psi:[0,+\infty)^2\to[0,+\infty)$ (we write $\Psi=\Psi(y_1,y_2)$) satisfying 
	\begin{enumerate} [label=\textup{($\Psi$\arabic*)}, start=1]	
		\item \label{Psi1}$\Psi(0,0)=0$;
		\item  \label{Psi2}$\Psi$ is bounded, Lipschitz, with $\nabla \Psi$ and  $\partial _{12}\Psi$ locally Lipschitz in $(0,+\infty)^2$;
		\item   \label{Psi3} for $i,j=1,2$ with $j\neq i$ there hold $\partial_i\Psi\geq y_i\partial_{ii}\Psi\vee 0$ and $\partial_{12}\Psi\leq y_i\partial_{iij}\Psi\wedge 0$ a.e. in $[0,+\infty)^2$;
		\item   \label{Psi4}$\displaystyle\sup_{y_1,y_2>0}(y_1+y_2)|\partial_{12}\Psi(y_1,y_2)|<+\infty$.
	\end{enumerate}
	These requirements are needed for the validity of the mathematical assumptions \ref{hyp:phi1}-\ref{hyp:phi5}; in order to have also the physical properties \ref{hyp:phi6}-\ref{hyp:phi9} we need to add:
	\begin{enumerate} [label=\textup{($\Psi$\arabic*)}, start=5]	
		\item \label{Psi5} $\nabla \Psi$ and $\partial_{12}\Psi$ vanish at infinity;
		\item  \label{Psi6} for $i,j=1,2$ with $j\neq i$ there holds $2\partial_{ii}\Psi\le y_j\partial_{iij}\Psi\wedge 0$ a.e. in $[0,+\infty)^2$ (or if $y_i\ge \bar z_i$, $y_j\ge 0$ in case of Remark~\ref{rmk:initelast}).
	\end{enumerate}
	
	Before writing the definition of $\Phi$, we introduce the following notation. Given $z_1,z_2\geq 0$, playing the role of history variables, we divide the space of possible openings into four regions:
	\begin{eqnarray*}
		&&R_1(z_1,z_2):=\{y_1\geq z_1,\ y_2\geq z_2\},\\
		&&R_2(z_1,z_2):=\{0\leq y_1< z_1,\ y_2\geq z_2\},\\
		&&R_3(z_1,z_2):=\{y_1\geq z_1,\ 0\leq y_2<z_2\},\\
		&&R_4(z_1,z_2):=\{0\leq y_1<z_1,\ 0\leq y_2<z_2\}.
	\end{eqnarray*}
	Observe that $R_1$ and $R_4$ represent the pure loading and the pure unloading phase, respectively; while $R_2$ and $R_3$ describe the mixed phases, in which one direction experiences loading while the other one unloading. 
	
	We then set $\Phi:[0,+\infty)^2{\times}[0,+\infty)^2\to \R$ as
	\begin{eqnarray*}
		\Phi(y_1,y_2,z_1,z_2):=
		\begin{cases}
			\Psi(y_1,y_2), & \text{if } (y_1,y_2)\in R_1(z_1,z_2),\\
			\displaystyle\Psi(z_1,y_2)-\frac{z_1}{2}\partial_1\Psi(z_1,y_2)\left(1-\left(\frac{y_1}{z_1}\right)^2\right), & \text{if } (y_1,y_2)\in R_2(z_1,z_2),\\
			\displaystyle\Psi(y_1,z_2)-\frac{z_2}{2}\partial_2\Psi(y_1,z_2)\left(1-\left(\frac{y_2}{z_2}\right)^2\right),  & \text{if } (y_1,y_2)\in R_3(z_1,z_2),\\
			\displaystyle\Psi(z_1,z_2)-\frac{z_1}{2}\partial_1\Psi(z_1,z_2)\left(1-\left(\frac{y_1}{z_1}\right)^2\right) & \text{}\\
			\displaystyle -\frac{z_2}{2}\partial_2\Psi(z_1,z_2)\left(1-\left(\frac{y_2}{z_2}\right)^2\right) & \text{}\\
			\displaystyle+\frac{z_1z_2}{4}\partial_{12}\Psi(z_1,z_2)\left(1-\left(\frac{y_1}{z_1}\right)^2\right)\left(1-\left(\frac{y_2}{z_2}\right)^2\right),
			& \text{if } (y_1,y_2)\in R_4(z_1,z_2).
		\end{cases}
	\end{eqnarray*}
	Let us briefly describe the structure of $\Phi$. In the pure loading zone $R_1$ it coincides with the loading density $\Psi$; in $R_2$ (and similarly in $R_3$), namely in the unloading phase for $y_1$, it behaves quadratically with respect to $y_1$, and the coefficients are chosen in order to have a smooth junction; finally, in the pure unloading zone $R_4$ the function $\Phi$ is quadratic in both directions $y_1$ and $y_2$.
	
	We actually observe that this expression provides the only candidate of potential energy whose partial derivatives (tensions) experience linear decay in the unloading phase with respect to the corrisponding variable. 
	
	We now check the validity of \ref{hyp:phi1} -- \ref{hyp:phi5}. Continuity of $\Phi$ in $[0,+\infty)^2\times[0,+\infty)^2$ and validity of hypotheses \ref{hyp:phi1} and \ref{hyp:phi4} are immediate. In order the check the remaining assumptions, it is useful to explicitly compute the partial derivatives of $\Phi$:
	\begin{subequations}\label{eq:gradient}
		\begin{eqnarray}
			\partial_{y_1}\Phi(y_1,y_2,z_1,z_2)=
			\begin{cases}
				\partial_{1}\Psi(y_1,y_2), & \text{if } (y_1,y_2)\in R_1(z_1,z_2),\\
				\displaystyle \partial_1\Psi(z_1,y_2)\frac{y_1}{z_1},& \text{if } (y_1,y_2)\in R_2(z_1,z_2),\\
				\displaystyle\partial_{1}\Psi(y_1,z_2)-\frac{z_2}{2}\partial_{12}\Psi(y_1,z_2)\left(1-\left(\frac{y_2}{z_2}\right)^2\right),  & \text{if } (y_1,y_2)\in R_3(z_1,z_2),\\
				\displaystyle \left(\partial_1\Psi(z_1,z_2)-\frac{z_2}{2}\partial_{12}\Psi(z_1,z_2)\left(1-\left(\frac{y_2}{z_2}\right)^2\right)\right)\frac{y_1}{z_1},
				& \text{if } (y_1,y_2)\in R_4(z_1,z_2),
			\end{cases}
		\end{eqnarray}
		\begin{eqnarray}
			\partial_{y_2}\Phi(y_1,y_2,z_1,z_2)=
			\begin{cases}
				\partial_{2}\Psi(y_1,y_2), & \text{if } (y_1,y_2)\in R_1(z_1,z_2),\\
				\displaystyle \partial_{2}\Psi(z_1,y_2)-\frac{z_1}{2}\partial_{12}\Psi(z_1,y_2)\left(1-\left(\frac{y_1}{z_1}\right)^2\right),  & \text{if } (y_1,y_2) \in R_2(z_1,z_2),\\
				\displaystyle \partial_2\Psi(y_1,z_2)\frac{y_2}{z_2},& \text{if } (y_1,y_2)\in R_3(z_1,z_2),\\
				\displaystyle \left(\partial_2\Psi(z_1,z_2)-\frac{z_1}{2}\partial_{12}\Psi(z_1,z_2)\left(1-\left(\frac{y_1}{z_1}\right)^2\right)\right)\frac{y_2}{z_2},
				& \text{if } (y_1,y_2)\in R_4(z_1,z_2),
			\end{cases}
		\end{eqnarray}
	\end{subequations}
	
	\begin{eqnarray*}
		\partial_{z_1}\Phi(y_1,y_2,z_1,z_2)=
		\begin{cases}
			0, & \text{if } (y_1,y_2)\in R_1(z_1,z_2),\\
			\displaystyle \frac{\partial_1\Psi(z_1,y_2)-z_1\partial_{11}\Psi(z_1,y_2)}{2}\left(1-\left(\frac{y_1}{z_1}\right)^2\right),  & \text{if } (y_1,y_2) \in R_2(z_1,z_2),\\
			0, & \text{if } (y_1,y_2)\in R_3(z_1,z_2),\\
			\displaystyle \left(\frac{\partial_1\Psi(z_1,z_2)-z_1\partial_{11}\Psi(z_1,z_2)}{2}-\frac{z_2}{4}(\partial_{12}\Psi(z_1,z_2)\right. & \text{}\\
			\displaystyle\left.-z_1\partial_{112}\Psi(z_1,z_2))\left(1-\left(\frac{y_2}{z_2}\right)^2\right)\right)
			\left(1-\left(\frac{y_1}{z_1}\right)^2\right),
			& \text{if } (y_1,y_2)\in R_4(z_1,z_2),
		\end{cases}
	\end{eqnarray*}
	\begin{eqnarray*}
		\partial_{z_2}\Phi(y_1,y_2,z_1,z_2)=
		\begin{cases}
			0, & \text{if } (y_1,y_2)\in R_1(z_1,z_2),\\
			0,  & \text{if } (y_1,y_2) \in R_2(z_1,z_2),\\
			\displaystyle \frac{\partial_2\Psi(y_1,z_2)-z_2\partial_{22}\Psi(y_1,z_2)}{2}\left(1-\left(\frac{y_2}{z_2}\right)^2\right), & \text{if } (y_1,y_2)\in R_3(z_1,z_2),\\
			\displaystyle \left(\frac{\partial_2\Psi(z_1,z_2)-z_2\partial_{22}\Psi(z_1,z_2)}{2}-\frac{z_1}{4}(\partial_{12}\Psi(z_1,z_2)\right. & \text{}\\
			\displaystyle\left.-z_2\partial_{122}\Psi(z_1,z_2))\left(1-\left(\frac{y_1}{z_1}\right)^2\right)\right)
			\left(1-\left(\frac{y_2}{z_2}\right)^2\right) 
			& \text{if } (y_1,y_2)\in R_4(z_1,z_2).
		\end{cases}
	\end{eqnarray*}
	From \ref{Psi3} we easily get $\partial_{y_i}\Phi\geq0$ and $\partial_{z_i}\Phi\geq0$ for $i=1,2$, hence \ref{hyp:phi5} holds. Furthermore, we deduce
	\[0=\Phi(0,0,0,0)\leq \Phi(y_1,y_2,z_1,z_2)\leq \Phi(y_1\vee z_1,y_2\vee z_2,z_1,z_2)=\Psi(y_1\vee z_1,y_2\vee z_2)\leq \sup_{[0,+\infty)^2}\Psi.\]
	Since the supremum in the last term is finite by \ref{Psi2}, we get that $\Phi$ is bounded and so in particular \ref{hyp:phi2} is fulfilled. It remains to verify \ref{hyp:phi3}. Let us check that $\partial_{y_1}\Phi$ is bounded uniformly with respect to $z$, the computation for $\partial_{y_2}\Phi$ being analogous. In $R_1$ and in $R_2$ the fact is trivial by exploiting \ref{Psi2}, while in $R_3$  we can estimate by using \ref{Psi4}, obtaining
	\begin{equation*}
		|\partial_{y_1}\Phi(y_1,y_2,z_1,z_2)|\le |\partial_1\Psi(y_1,z_2)|+z_2|\partial_{12}\Psi(y_1,z_2)|\le C.
	\end{equation*}
	Similar computations yield the result also in $R_4$, and we conclude.
	
	Let us now check \ref{hyp:phi6}-\ref{hyp:phi9}, assuming in addition \ref{Psi5} and \ref{Psi6}. Notice that we already proved \ref{hyp:phi6}, while we observe that \ref{hyp:phi9} is automatically satisfied whenever \ref{Psi5} is in force. Moreover, \ref{hyp:phi7} follows from \ref{Psi3} and \ref{Psi6} by observing that the partial derivatives $\partial_{y_i}\Phi$ are continuous and by exploiting the expression
	\begin{eqnarray*}
		\partial_{y_1,y_1}\Phi(y_1,y_2,z_1,z_2)=
		\begin{cases}
			\partial_{11}\Psi(y_1,y_2), & \text{if } (y_1,y_2)\in R_1(z_1,z_2),\\
			\displaystyle \partial_1\Psi(z_1,y_2)\frac{1}{z_1},& \text{if } (y_1,y_2)\in R_2(z_1,z_2),\\
			\displaystyle\partial_{11}\Psi(y_1,z_2)-\frac{z_2}{2}\partial_{112}\Psi(y_1,z_2)\left(1-\left(\frac{y_2}{z_2}\right)^2\right),  & \text{if } (y_1,y_2)\in R_3(z_1,z_2),\\
			\displaystyle \left(\partial_1\Psi(z_1,z_2)-\frac{z_2}{2}\partial_{12}\Psi(z_1,z_2)\left(1-\left(\frac{y_2}{z_2}\right)^2\right)\right)\frac{1}{z_1},
			& \text{if } (y_1,y_2)\in R_4(z_1,z_2),
		\end{cases}
	\end{eqnarray*}
	and analogously for $\partial_{y_2,y_2}\Phi$. Finally, \ref{Psi3} yields that the mixed derivative
	\begin{eqnarray*}
		\partial_{y_1,y_2}\Phi(y_1,y_2,z_1,z_2)=
		\begin{cases}
			\partial_{12}\Psi(y_1,y_2), & \text{if } (y_1,y_2)\in R_1(z_1,z_2),\\
			\displaystyle \partial_{12}\Psi(z_1,y_2)\frac{y_1}{z_1},& \text{if } (y_1,y_2)\in R_2(z_1,z_2),\\
			\displaystyle\partial_{12}\Psi(y_1,z_2)\frac{y_2}{z_2},  & \text{if } (y_1,y_2)\in R_3(z_1,z_2),\\
			\displaystyle \partial_{12}\Psi(z_1,z_2)\frac{y_1}{z_1}\frac{y_2}{z_2},
			& \text{if } (y_1,y_2)\in R_4(z_1,z_2),
		\end{cases}
	\end{eqnarray*}
	is nonnegative, whence also \ref{hyp:phi8} holds true.
	
	\subsection{Construction of the non-potential tension $\mc T$} We adopt the same notations of the previous section. Consider a loading tension $\mc S=(\mc S_1,\mc S_2)\colon [0,+\infty)^2\to \R^2$ satisfying
	\begin{enumerate} [label=\textup{($\mc S$\arabic*)}, start=1]	
		\item  \label{S1}$\mc S$ is continuous and bounded in $[0,+\infty)^2$.
	\end{enumerate}		
	
	The tension $\mc T=(\mc T_1,\mc T_2)\colon [0,+\infty)^2\times [0,+\infty)^2\to \R^2$ is then defined as
	\begin{subequations}\label{eq:tensions}
		\begin{eqnarray}
			\mc T_1(y_1,y_2,z_1,z_2)=
			\begin{cases}
				\mc S_1(y_1,y_2), & \text{if } (y_1,y_2)\in R_1(z_1,z_2),\\
				\displaystyle \mc S_1(z_1,y_2)\frac{y_1}{z_1},& \text{if } (y_1,y_2)\in R_2(z_1,z_2),\\
				\displaystyle \mc S_1(y_1,z_2),  & \text{if } (y_1,y_2)\in R_3(z_1,z_2),\\
				\displaystyle \mc S_1(z_1,z_2)\frac{y_1}{z_1},
				& \text{if } (y_1,y_2)\in R_4(z_1,z_2),
			\end{cases}
		\end{eqnarray}
		\begin{eqnarray}
			\mc T_2(y_1,y_2,z_1,z_2)=
			\begin{cases}
				\mc S_2(y_1,y_2), & \text{if } (y_1,y_2)\in R_1(z_1,z_2),\\
				\displaystyle \mc S_2(z_1,y_2),  & \text{if } (y_1,y_2) \in R_2(z_1,z_2),\\
				\displaystyle \mc S_2(y_1,z_2)\frac{y_2}{z_2},& \text{if } (y_1,y_2)\in R_3(z_1,z_2),\\
				\displaystyle \mc S_2(z_1,z_2)\frac{y_2}{z_2},
				& \text{if } (y_1,y_2)\in R_4(z_1,z_2).
			\end{cases}
		\end{eqnarray}
	\end{subequations}
	
	The validity of \ref{hyp:T1} and \ref{hyp:T2} is immediate. One may also directly verify the physical requirements \ref{hyp:T3}-\ref{hyp:T6} assuming in addition that
	\begin{enumerate} [label=\textup{($\mc S$\arabic*)}, start=2]	
		\item  \label{S2}$\mc S$ is valued in $[0,+\infty)^2$ and vanishes at infinity;
		\item   \label{S3} $\mc S_1$ and $\mc S_2$ are nonincreasing in each component ($\mc S_i$ nonincreasing with respect to $y_i$ only in $[\bar z_i,+\infty)$ in case of an intrinsic cohesive model).
	\end{enumerate}
	
	If $\mc S=\nabla\Psi$ for some suitable density $\Psi$, notice that the expressions \eqref{eq:tensions} coincide with \eqref{eq:gradient} in the (not interesting) case of uncoupled directions, namely when $\partial_{12}\Psi\equiv 0$, or equivalently $\Psi(y_1,y_2)=\psi_1(y_1)+\psi_2(y_2)$. This is not a coincidence, indeed it is easy to check that the pair $(\mc T_1,\mc T_2)$ defines a gradient (with respect to $y$) just in this occurence. 
	
	\subsection{Examples of loading density $\Psi$ }
	
	We now present some specific instances of functions $\Psi$. They all fall within the following abstract structure. 
	
	Let us consider two one-dimensional cohesive densities (see \cite{BonCavFredRiva,Riv}) $\psi_i\in C^{1,1}_{\rm loc}([0,+\infty))$, $i=1,2$, namely each $\psi_i$ is bounded, Lipschitz, nondecreasing, and satisfies $\psi_i(0)=0$ and 
	\begin{subequations}\label{eq:psiassumptions}
		\begin{align}
			&\psi_i'(y)-y\psi''_i(y)\geq0,\label{psia}\\
			&\sup_{y>0}y\psi_i'(y)<+\infty \label{psib}.
		\end{align}
	\end{subequations}
	
	Assuming that $\sup \psi_i=1$, we also consider a nonnegative function $F\in C^{2,1}([0,1]^2)$ such that $F(0,0)=0$, $\partial_i F\geq0$, $\partial_{ii}F\leq0$, $\partial_{12}F\leq0$, $\partial_{iij} F\ge 0$ in $[0,1]^2$ for $i,j=1,2$, $i\neq j$.
	
	The density $\Psi$ is now defined as
	\begin{equation}\label{eq:examplePsiF}
		\Psi(y_1,y_2):=F(\psi_1(y_1),\psi_2(y_2)).
	\end{equation}
	
	Let us now check that it satisfies conditions \ref{Psi1}--\ref{Psi4}. It is trivial to see that \ref{Psi1} and \ref{Psi2} hold true, while by simple computations we deduce
	\begin{equation}\label{eq:comp1}
		\partial_i\Psi(y_1,y_2)=\partial_i F(\psi_1(y_1),\psi_2(y_2))\psi'_i(y_i)\ge 0,
	\end{equation}
	since it is the product of nonnegative terms. Analogously we obtain
	\begin{equation}\label{eq:comp2}
		\partial_{12}\Psi(y_1,y_2)=\partial_{12} F(\psi_1(y_1),\psi_2(y_2))\psi'_1(y_1)\psi'_2(y_2)\le 0.
	\end{equation}
	Moreover, we have
	\begin{align}\label{eq:comp3}
		&\partial_i\Psi(y_1,y_2)-y_i \partial_{ii}\Psi(y_1,y_2)\nonumber\\
		=&\partial_i F(\psi_1(y_1),\psi_2(y_2))\Big(\psi_i'(y_i)-y_i \psi''_i(y_i)\Big)-y_i \psi'_i(y_i)^2\partial_{ii} F(\psi_1(y_1),\psi_2(y_2))\ge 0,
	\end{align}
	where we used \eqref{psia} and the assumptions on $F$. Similarly, there holds
	\begin{align}\label{eq:comp4}
		&\partial_{12}\Psi(y_1,y_2)-y_i \partial_{iij}\Psi(y_1,y_2)\nonumber\\
		=& \psi_j'(y_j)\Big((\psi_i'(y_i)-y_i \psi''_i(y_i))\partial_{12} F(\psi_1(y_1),\psi_2(y_2))-y_i(\psi_i'(y_i))^2\partial_{iij} F(\psi_1(y_1),\psi_2(y_2))\Big)\le 0.
	\end{align}
	Combining \eqref{eq:comp1}, \eqref{eq:comp2}, \eqref{eq:comp3} and \eqref{eq:comp4} we finally infer \ref{Psi3}. By exploiting \eqref{psib} we also show \ref{Psi4}:
	\begin{align*}
		&\sup_{y_1,y_2>0}(y_1+y_2)|\partial_{12}\Psi(y_1,y_2)|= \sup_{y_1,y_2>0}(y_1+y_2)|\partial_{12} F(\psi_1(y_1),\psi_2(y_2))|\psi'_1(y_1)\psi'_2(y_2)\\
		\le& \max_{[0,1]^2}|\partial_{12} F|\left(\sup \psi_2' \sup_{y_1\ge0}y_1\psi_1'(y_1)+\sup \psi_1' \sup_{y_2\ge0}y_2\psi_2'(y_2)\right)<+\infty.
	\end{align*}
	If in addition $F$ satisfies
	\begin{equation}\label{eq:Fvanishing}
		\partial_1 F(\xi_1,1)=\partial_2 F(1,\xi_2)=0,\qquad\text{for all }(\xi_1,\xi_2)\in [0,1]^2,
	\end{equation}
	then also \ref{Psi5} can be directly obtained.
	
	The validity of \ref{Psi6} can be instead deduced if the functions $\psi_i$ are concave in $[0,+\infty)$ (or, in view of Remark~\ref{rmk:initelast}, in $[\bar z_i,+\infty)$). Indeed, recalling the assumptions on $F$, in this case we infer
	\begin{align*}
		\partial_{ii}\Psi(y_1,y_2)&= \psi_i''(y_i)\partial_i F(\psi_1(y_1),\psi_2(y_2))+(\psi_i'(y_i))^2\partial_{ii} F(\psi_1(y_1),\psi_2(y_2))\\
		&\le 0\le \frac{y_j}{2}\psi_j'(y_j)\Big(\psi_i''(y_i)\partial_{12} F(\psi_1(y_1),\psi_2(y_2))+(\psi_i'(y_i))^2\partial_{iij} F(\psi_1(y_1),\psi_2(y_2))\Big) \\
		&=\frac{y_j}{2}\partial_{iij}\Psi(y_1,y_2).
	\end{align*}
	
	The simplest, although effective, choice of auxiliary function $F$ satisfying the previous assumptions is given by
	\begin{equation}\label{eq:F}
		F(\xi_1,\xi_2):=\Phi_1\xi_1+\Phi_2\xi_2-\alpha \xi_1\xi_2,
	\end{equation}
	where $\Phi_1,\Phi_2\ge 0$ are nonnegative constants representing complete delamination energies, while the parameter $\alpha$ satisfies $0\le \alpha\le \Phi_1\wedge \Phi_2$. However, observe that \eqref{eq:Fvanishing}, needed to preserve a realistic behaviour, is in force if and only if
	\begin{equation}\label{eq:phi12}
		\alpha=\Phi_1=\Phi_2.
	\end{equation}
	This represents a first limitation of potential-based models, indeed the above constraint represents materials whose delamination energies are equal in both directions.
	
	We now consider examples of one-dimensional densities fulfilling the required assumptions.
	
	\subsubsection*{Negative exponentials} The first example is given by 
	\begin{equation*}
		\psi_{\rm exp}(y)=1-e^{-\rho y},
	\end{equation*}
	with $\rho>0$. It features an infinite delamination opening, namely it reaches its supremum just asymptotically as $y\to +\infty$. 
	
	\subsubsection*{Polynomial behaviour} A second simple example with finite delamination opening $\delta>0$ is the cubic law
	\begin{equation*}
		\psi_{\rm cub}(y)=\begin{cases}\displaystyle \frac y\delta \left(\left(\frac y\delta\right)^2-3 \frac y\delta+3\right), & \text{if }y\in [0,\delta),\\
			1,& \text{if }y\ge \delta.
		\end{cases}
	\end{equation*}
	
	\subsubsection*{Intrinsic densities}
	
	Starting from \emph{concave} densities like the previous two examples one can always construct a whole family of new densities, featuring an initial quadratic (i.e. elastic) behaviour and thus suitable to describe the so-called intrinsic cohesive models. Given a concave function $\psi\in C^{1,1}_{\rm loc}([0,+\infty))$ bounded, Lipschitz, nondecreasing, satisfying $\psi(0)=0$ and \eqref{psib}, for any parameter $\eps>0$ it is enough to define
	\begin{equation*}
		\psi_\eps(y):=\begin{cases}
			\displaystyle\frac{y^2}{2\eps},&\text{if }y\in [0,\bar z_\eps],\\
			\displaystyle \psi(y)-\psi(\bar z_\eps)+\frac{\bar z^2_\eps}{2\eps},&\text{if }y\in(\bar z_\eps,+\infty),
		\end{cases}
	\end{equation*}
	where $\bar z_\eps>0$ is the unique positive number satisfying $\bar z_\eps=\eps\psi'(\bar z_\eps)$. Notice that conditions \eqref{eq:psiassumptions}, as well as the above assumptions, are satisfied by $\psi_\eps$; moreover, it is concave in $[\bar z_\eps,+\infty)$. We actually observe that, in order to have $\sup\psi_\eps=1$, one shall rescale the function with the constant $(1-\psi(\bar z_\eps)+{\bar z^2_\eps}/{2\eps})^{-1}$.
	
	\subsubsection*{PPR model} We now discuss a more involved example introduced in \cite{PPR}--and called PPR from the name of the authors--for an analogous model of fracture. We first introduce the intrinsic cohesive zone model, characterized by an initial elastic behaviour. For $i=1,2$, let us consider the parameters $\alpha_i>1$, $\Phi_i,\sigma_i>0$ and $\lambda_i\in \left(0,\frac{1}{\sqrt{\alpha_i}}\right)$, representing shape index (characterizing material softening responses), complete delamination energy, cohesive strenght, and initial slope indicator in direction $i$, respectively. We then define the constants
	\begin{equation}\label{eq:mi}
		m_i:=\frac{\alpha_i(\alpha_i-1)\lambda_i^2}{1-\alpha_i\lambda_i^2}>0,
	\end{equation}
	the final slip width (or delamination opening)
	\begin{equation}\label{eq:delta_i}
		\delta_i:=\frac{\Phi_i}{\sigma_i}\alpha_i\lambda_i(1-\lambda_i)^{\alpha_i-1}\left(1+\frac{\alpha_i}{m_i}\right)\left(1+\lambda_i\frac{\alpha_i}{m_i}\right)^{m_i-1}>0,
	\end{equation}
	and the energy constants
	\begin{equation}\label{eq:Gammai}
		\Gamma_1:=\begin{cases}
			\displaystyle	-\Phi_1\left(\frac{\alpha_1}{m_1}\right)^{m_1},&\text{if }\Phi_1\ge \Phi_2,\\
			\displaystyle	\left(\frac{\alpha_1}{m_1}\right)^{m_1},&\text{if }\Phi_1< \Phi_2,
		\end{cases}\qquad \Gamma_2:=\begin{cases}
			\displaystyle	\left(\frac{\alpha_2}{m_2}\right)^{m_2},&\text{if }\Phi_1\ge \Phi_2,\\
			\displaystyle	-\Phi_2\left(\frac{\alpha_2}{m_2}\right)^{m_2},&\text{if }\Phi_1< \Phi_2. 
		\end{cases}
	\end{equation}
	The loading density in the PPR model is then defined as (see (8) in \cite{PPR}) 
	\begin{equation*}
		\begin{aligned}
			\Psi_{PPR}(y_1,y_2):=\Phi_1\wedge \Phi_2+&\left[\Gamma_1\left(\left(1-\frac{y_1}{\delta_1}\right)^+\right)^{\alpha_1}\left(\frac{m_1}{\alpha_1}+\frac{y_1}{\delta_1}\right)^{m_1}+(\Phi_1-\Phi_2)^+\right]\times\\
			&\left[\Gamma_2\left(\left(1-\frac{y_2}{\delta_2}\right)^+\right)^{\alpha_2}\left(\frac{m_2}{\alpha_2}+\frac{y_2}{\delta_2}\right)^{m_2}+(\Phi_2-\Phi_1)^+\right].
		\end{aligned}
	\end{equation*}
	
	We observe that, after some simple manipulation, we can rewrite the PPR density in the form \eqref{eq:examplePsiF}, with $F$ as in \eqref{eq:F}. Indeed we have
	\begin{equation}\label{eq:PPrrewritten}
		\Psi_{PPR}(y_1,y_2)=\Phi_1\psi_1(y_1)+\Phi_2\psi_2(y_2)-\Phi_1\vee\Phi_2 \psi_1(y_1)\psi_2(y_2),
	\end{equation}
	where the one-dimensional densities are given by 
	\begin{equation}\label{eq:1dPPR}
		\psi_i(y_i):=1-\left(\left(1-\frac{y_i}{\delta_i}\right)^+\right)^{\alpha_i}\left(1+\frac{\alpha_i}{m_i}\frac{y_i}{\delta_i}\right)^{m_i}.
	\end{equation}
	
	\begin{rmk}\label{rmk:Phi12}
		Comparing \eqref{eq:PPrrewritten} with \eqref{eq:F}, we observe that the constraint $\alpha\le \Phi_1\wedge \Phi_2$ here yields $\Phi_1=\Phi_2$ (see also \eqref{eq:phi12}), namely we are forced to consider equal delamination energy in both directions in order to keep the potential-based structure of the model. This drawback was already observed in \cite[Section 2.2]{PPR} (see also \cite{ParkPaulino}), and it is an intrinsic limitation of potential-based models, which can be overcome by means of non potential-based ones.
	\end{rmk}
	
	We now show that the one-dimensional PPR densities fulfil our assumptions. 
	\begin{lemma}\label{lemma:PPR}
		If $\lambda_i\le \frac{1}{\sqrt{2\alpha_i-1}}$, then the densities defined in \eqref{eq:1dPPR} are nonnegative and satisfy the following properties:
		\begin{itemize}
			\item $\psi_i(0)=\psi_i'(0)=0$;
			\item $\psi_i\in C^{1,1}([0,+\infty))$ is nondecreasing and bounded;
			\item assumption \eqref{eq:psiassumptions} holds true;
			\item $\psi_i$ is concave in $[\delta_i\lambda_i,+\infty)$.
		\end{itemize}
	\end{lemma}
	\begin{proof}
		By computing the derivatives of $\psi_i$ we obtain
		\begin{align*}
			&\psi_i'(y_i)=\frac{\alpha_i}{\delta_i}\left(1+\frac{\alpha_i}{m_i}\right)\frac{y_i}{\delta_i}\left(\left(1-\frac{y_i}{\delta_i}\right)^+\right)^{\alpha_i-1}\left(1+\frac{\alpha_i}{m_i}\frac{y_i}{\delta_i}\right)^{m_i-1},\\
			&\psi_i''(y_i)=\frac{\alpha_i}{\delta_i^2}\left(1+\frac{\alpha_i}{m_i}\right)\left(\left(1-\frac{y_i}{\delta_i}\right)^+\right)^{\alpha_i-2}\left(1+\frac{\alpha_i}{m_i}\frac{y_i}{\delta_i}\right)^{m_i-2}\left(1-\frac{\alpha_i}{m_i}(\alpha_i+m_i-1)\frac{y_i^2}{\delta_i^2}\right),
		\end{align*}
		whence one easily deduces the regularity of $\psi_i$ and the fact that it is nondecreasing. Moreover, $\psi_i''$ is nonpositive in $\left[\delta_i\sqrt{\frac{m_i}{\alpha_i(\alpha_i+m_i-1)}},+\infty\right)$, namely in $[\delta_i\lambda_i,+\infty)$ by means of \eqref{eq:mi}, and so $\psi_i$ is concave therein.
		
		Using the obvious property $\psi_i(0)=0$, one then infers that $\psi_i$ is nonnegative and bounded (with $\max \psi_i=1$). Also property \eqref{psib} easily follows since $\psi_i'$ is supported in $[0,\delta_i]$, so we are left to check \eqref{psia}. To this aim we compute
		\begin{equation*}
			\psi_i'(y_i)- y_i\psi_i''(y_i)=\frac{\alpha_i}{\delta_i}\left(\!1{+}\frac{\alpha_i}{m_i}\right)\frac{y_i^2}{\delta_i^2}\left(\!\!\left(1{-}\frac{y_i}{\delta_i}\right)^+\!\right)^{\alpha_i-2}\!\!\!\!\left(1{+}\frac{\alpha_i}{m_i}\frac{y_i}{\delta_i}\right)^{m_i-2}\!\!\!\left(\frac{\alpha_i}{m_i}(\alpha_i{+}m_i{-}2)\frac{y_i}{\delta_i}{+}\frac{\alpha_i}{m_i}{-}1\right).
		\end{equation*}
		Observing that, by the expression \eqref{eq:mi}, the assumption $\lambda_i\le (2\alpha_i-1)^{-1/2}$ yields $\alpha_i\ge m_i$, and since under our set of assumptions there holds $\alpha_i(\alpha_i+m_i-1)\ge m_i$, we now infer that the last term within brackets in the above expression is nonnegative (we recall that it is enough to check it for $y_i\in [0,\delta_i]$). This yields \eqref{psia} and we conclude. 
	\end{proof}
	
	Sending the initial slope indicators $\lambda_i$ to zero we recover the extrinsic cohesive zone version of the PPR model, corresponding to a completely anelastic process. By the expressions \eqref{eq:mi}, \eqref{eq:delta_i}, \eqref{eq:Gammai}, after simple computations as $\lambda_i\to 0$ we deduce
	\begin{align*}
		& m_i\to 0,\\
		&\delta_i\to \bar\delta_i:= \frac{\Phi_i}{\sigma_i}\alpha_i,\\
		&\Gamma_1\to \bar\Gamma_1:=\begin{cases}
			\displaystyle	-\Phi_1,&\text{if }\Phi_1\ge \Phi_2,\\
			\displaystyle	1,&\text{if }\Phi_1< \Phi_2,
		\end{cases}\qquad \Gamma_2\to\bar\Gamma_2:=\begin{cases}
			\displaystyle	1,&\text{if }\Phi_1\ge \Phi_2,\\
			\displaystyle	-\Phi_2,&\text{if }\Phi_1< \Phi_2,
		\end{cases}
	\end{align*}
	whence we obtain $\Psi_{PPR}(y_1,y_2)\to \bar\Psi_{PPR}(y_1,y_2)$, where 
	\begin{equation*}
		\bar\Psi_{PPR}(y_1,y_2):=\Phi_1\wedge \Phi_2+\left[\bar\Gamma_1\left(\left(1-\frac{y_1}{\bar\delta_1}\right)^+\right)^{\alpha_1}\!\!\!\!+(\Phi_1{-}\Phi_2)^+\right]\!\!\left[\bar\Gamma_2\left(\left(1-\frac{y_2}{\bar\delta_2}\right)^+\right)^{\alpha_2}\!\!\!\!+(\Phi_2{-}\Phi_1)^+\right]\!.
	\end{equation*}
	Arguing as before we can rewrite $\bar\Psi_{PPR}$ as
	\begin{equation}\label{eq:PPrrewritten2}
		\bar\Psi_{PPR}(y_1,y_2)=\Phi_1\bar\psi_1(y_1)+\Phi_2\bar\psi_2(y_2)-\Phi_1\vee\Phi_2 \bar\psi_1(y_1)\bar\psi_2(y_2),
	\end{equation}
	where the extrinsic one-dimensional densities are given by 
	\begin{equation}\label{eq:1dPPR2}
		\bar\psi_i(y_i):=1-\left(\left(1-\frac{y_i}{\bar\delta_i}\right)^+\right)^{\alpha_i}.
	\end{equation}
	Again, we stress that Remark~\ref{rmk:Phi12} still applies.
	
	We conclude by showing that also in the extrinsic case the one-dimensional densities fulfil our assumptions.
	\begin{lemma}
		The densities defined in \eqref{eq:1dPPR2} are nonnegative and satisfy the following properties:
		\begin{itemize}
			\item $\bar\psi_i(0)=0$;
			\item $\bar\psi_i\in C^{1,1}([0,+\infty))$ is nondecreasing, bounded and concave;
			\item assumption \eqref{eq:psiassumptions} holds true;
			\item $\bar \psi_i$ is concave.
		\end{itemize}
	\end{lemma}
	\begin{proof}
		The result follows by arguing as in Lemma~\ref{lemma:PPR} once we compute
		\begin{align*}
			&\bar\psi_i'(y_i)=\frac{\alpha_i}{\bar\delta_i}\left(\left(1-\frac{y_i}{\bar\delta_i}\right)^+\right)^{\alpha_i-1}\ge 0,\\
			&\bar\psi_i''(y_i)=-\frac{\alpha_i(\alpha_i-1)}{\bar\delta_i^2}\left(\left(1-\frac{y_i}{\bar\delta_i}\right)^+\right)^{\alpha_i-2}\le 0.
		\end{align*}
	\end{proof}
	
	\subsection{Examples of loading tension $\mc S$}
	In non potential-based models a good choice of loading tension is given by 
	\begin{equation}\label{eq:S}
		\mc S(y_1,y_2):=\nabla\Psi(y_1,y_2) \bm{\vee}0_2=(\partial_1\Psi(y_1,y_2)^+, \partial_2\Psi(y_1,y_2)^+),
	\end{equation}
	where $\Psi$ has again the form \eqref{eq:examplePsiF}, but with weaker requirements than before. Here, the function $F$ is just assumed to be of class $C^1$, with no constraints on the sign of its derivatives; while the one-dimensional densities $\psi_i$ do not need to fulfil \eqref{eq:psiassumptions} anymore. Condition \ref{S1} is now a direct consequence of the just listed regularity assumptions.
	
	Conditions \ref{S2} and \ref{S3} require slightly stronger assumptions, as expected. The former is fulfilled whenever $\psi_i'$ vanish as $y_i\to +\infty$ and
	\begin{equation}\label{eq:Fder}
		\partial_1 F(\xi_1,1)\le 0,\quad\text{and}\quad\partial_2 F(1,\xi_2)\le 0,\qquad\text{for all }(\xi_1,\xi_2)\in [0,1]^2.
	\end{equation}
	The latter is instead implied by assuming that $\psi_i$ are concave in $[\bar z_i,+\infty)$, and that $F$ is of class $C^2$ with $\partial_{12}F\le 0$ and $\partial_{ii}F\le 0$. To show it, as before it is enough to compute the derivatives of $S_i$ and verify that they are nonpositive.
	
	The explicit form \eqref{eq:F} of $F$ is included in this setting (in particular \eqref{eq:Fder} is in force) if and only if the parameter $\alpha$ satisfies $\alpha\ge \Phi_1\vee\Phi_2$. For instance, observe that the PPR densities \eqref{eq:PPrrewritten} and \eqref{eq:PPrrewritten2} fit in the non-potential framework even if, differently than the potential-based model (see Remark~\ref{rmk:Phi12}), the two delamination energies $\Phi_1$ and $\Phi_2$ are different.
	
	This fact shows the flexibility of non-variational models with respect to variational ones, in the mixed-mode case.
	
	\section{Representative instances: non trivial loading/unloading paths}\label{sec:numeric}
	In this section, we illustrate the response of a cohesive interface under non trivial loading/unloading/reloading paths in dimension $d=2$ with the prototypical cohesive variable $\mathfrak g(\delta)=(|\delta_1|,|\delta_2|)$. In particular, the respective
	opening separations are assumed to follow the relations:
	\begin{equation*}
		y_1(t)=|a_1\sin(b_1 t)|,\quad y_2(t)=|a_2\sin(b_2 t)|,
	\end{equation*}
	where $t$ is a time-like parameter. We consider the variational intrinsic PPR energy density \eqref{eq:PPrrewritten} and the corresponding non potential-based law induced by \eqref{eq:S}. Without loss of generality the following parameter values are taken: $\alpha_1=\alpha_2=2$, $\sigma_1=\sigma_2=2$ MPa and $\lambda_1	=\lambda_2=0.2$.
	
	The following situations are investigated:
	\begin{itemize}
		\item CASE 1: equal displacement slip evolution ($a_1=a_2=1$, $b_1=b_2=0.2$) and equal energy values $\Phi_1=\Phi_2=2$ N/m (see Fig. \ref{fig_Equal_all}a); 
		\item CASE 2: equal displacement slip values with different phases ($a_1=a_2=1$, $b_1=0.2,\, b_2=0.3$) and equal energy values $\Phi_1=\Phi_2=2$ N/m (see Fig. \ref{fig_Different_v}a);
		\item CASE 3: different displacement slip values with different phases ($a_1=1,\, a_2=3$, $b_1=0.2,\,b_2=0.3$) and different energy values $\Phi_1=6, \, \Phi_2=2$ N/m (see Fig. \ref{fig_All_different}a);
		\item CASE 4: different displacement slip values with different phases ($a_1=1,\, a_2=0.5$, $b_1=0.125,\,b_2=0.4$) but unloading of $y_1$ at fixed $y_2$ value and equal energy values $\Phi_1=\Phi_2=2$ N/m (see Fig. \ref{fig_PPR_funziona}a).
	\end{itemize}
	
	For each case the following plots are given:   evolution of displacement slip values $y_1, y_2$ and history variables $z_1, z_2$, energy evolution (only for the variational model) and traction-displacement slip relations.
	
	Although the energy evolutions of the variational model depicted in Figs. \ref{fig_Equal_all}b, \ref{fig_Different_v}b, \ref{fig_All_different}b seem consistent, the same cannot be said of its derivatives, which from the engineering point of view have a crucial meaning. In all cases, the computed traction–separation relations during the first loading path for the potential-based model and the non potential-based law are equivalent. In CASE 1, the unloading/reloading path is nonlinear for the variational model, see Figs. \ref{fig_Equal_all}c,d, while it is fully linear for the non-potential model Figs. \ref{fig_Equal_all_N}a,b. For CASE 2, the unloading/reloading path of the potential-based model reported in Figs. \ref{fig_Different_v}c,d reveals significant deviations from the expected results as the one obtained with the non potential-based law as depicted in Figs. \ref{fig_Different_v_N}a,b. The problem is further exacerbated in CASE 3, where the potential-based model gives a totally nonphysical response as reported in Figs. \ref{fig_All_different}c,d differently from the non potential-based model that does not allow the change of sign of the stress under partial unloading conditions, see Figs. \ref{fig_All_different_N}a,b.
	Only in the very special CASE 4 the variational model provides a physically reasonable result in the case of partial unloading as illustrated in Fig. \ref{fig_PPR_funziona}.

	\begin{figure}[htpb]
		\centering
		\includegraphics[width=0.95\textwidth]{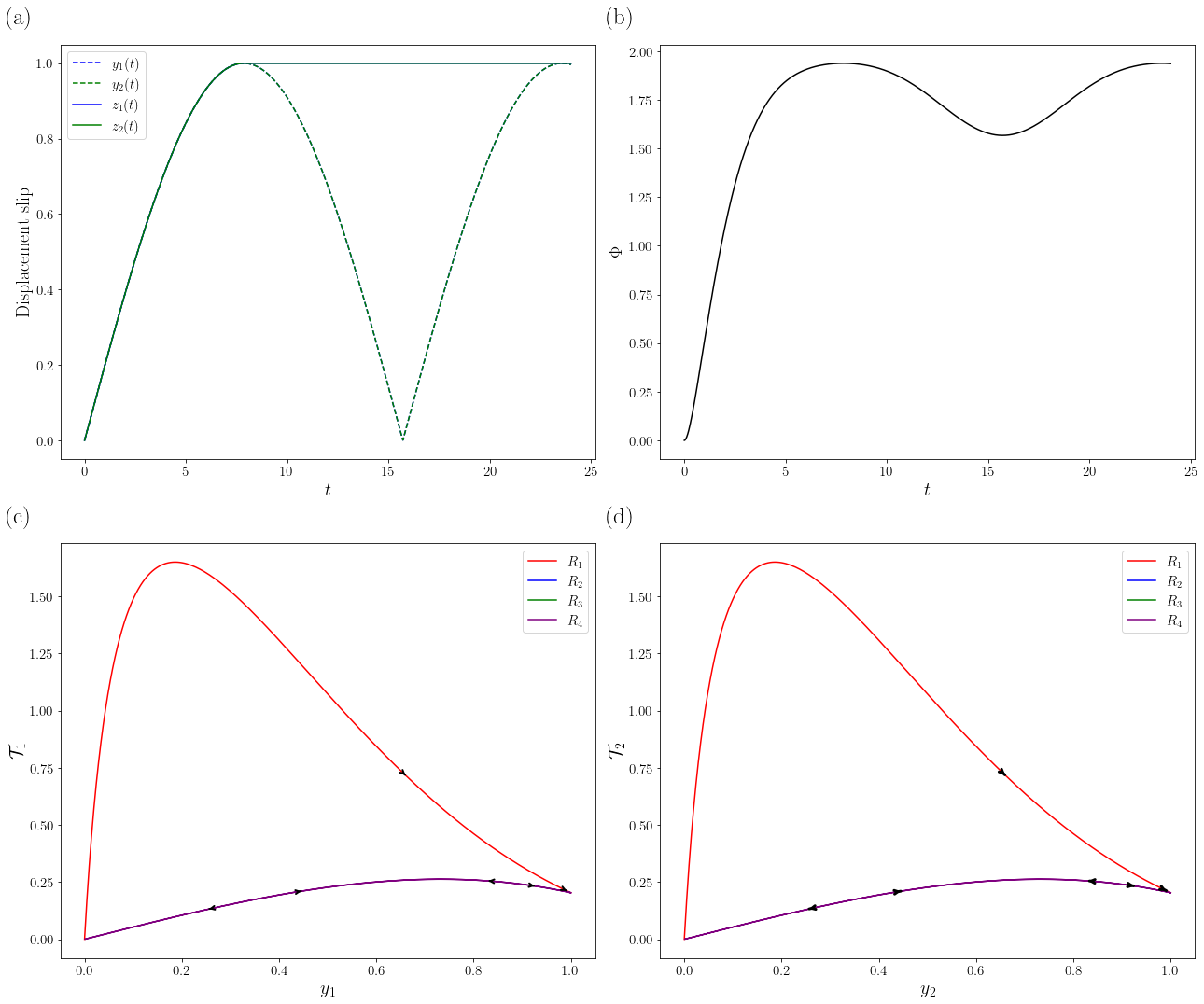}
		\caption{CASE 1: potential-based model. a) Evolution of displacement slip values $y_1, y_2$ and history variables $z_1, z_2$, b) Energy evolution, c) Traction-displacement slip relation $\mc T_1 - y_1$, d) Traction-displacement slip relation $\mc T_2 - y_2$.} 
		\label{fig_Equal_all} 
	\end{figure}
	
	\begin{figure}[htpb]
		\centering
		\includegraphics[width=0.95\textwidth]{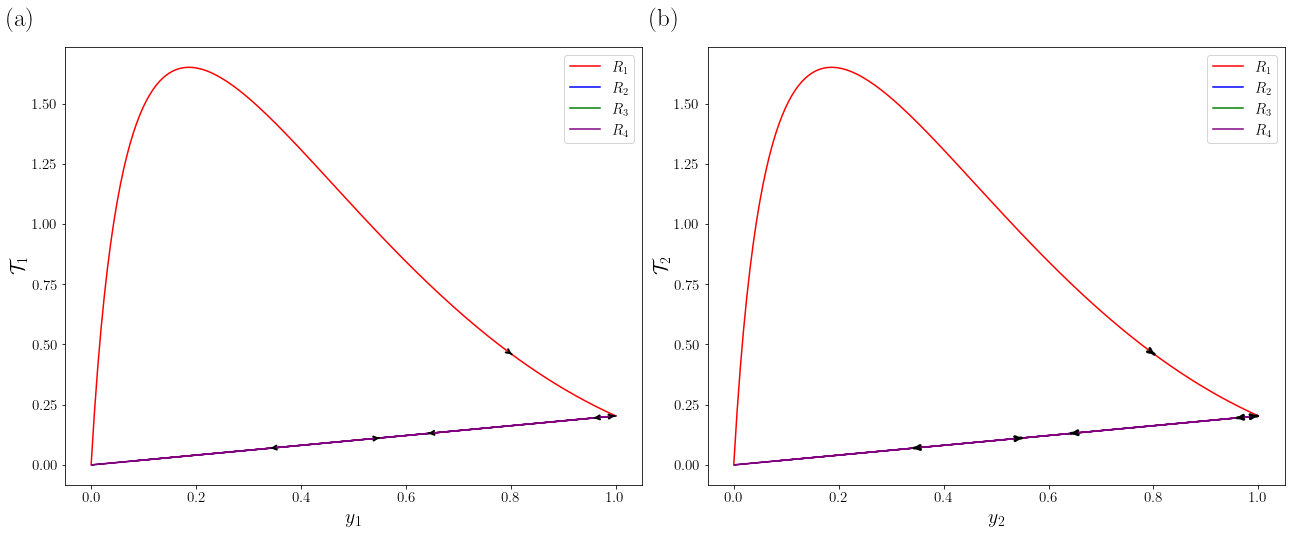}
		\caption{CASE 1: non potential-based model.  a) Traction-displacement slip relation $\mc T_1 - y_1$, b) Traction-displacement slip relation $\mc T_2 - y_2$.} 
		\label{fig_Equal_all_N} 
	\end{figure}

	\begin{figure}[htpb]
		\centering
		\includegraphics[width=0.95\textwidth]{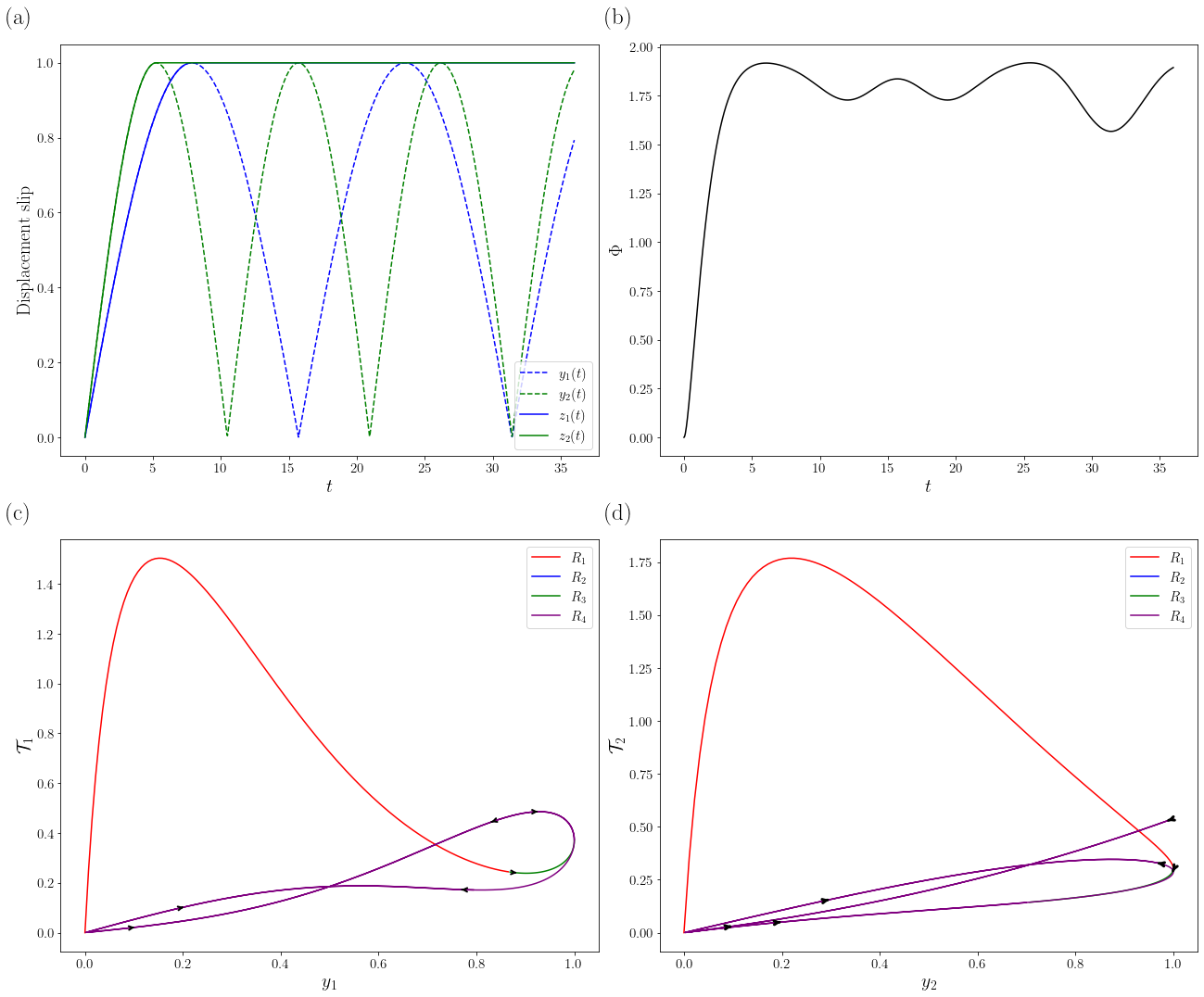}
		\caption{CASE 2: potential-based model. a) Evolution of displacement slip values $y_1, y_2$ and history variables $z_1, z_2$, b) Energy evolution, c) Traction-displacement slip relation $\mc T_1 - y_1$, d) Traction-displacement slip relation $\mc T_2 - y_2$.} 
		\label{fig_Different_v} 
	\end{figure}
	
	\begin{figure}[htpb]
		\centering
		\includegraphics[width=0.95\textwidth]{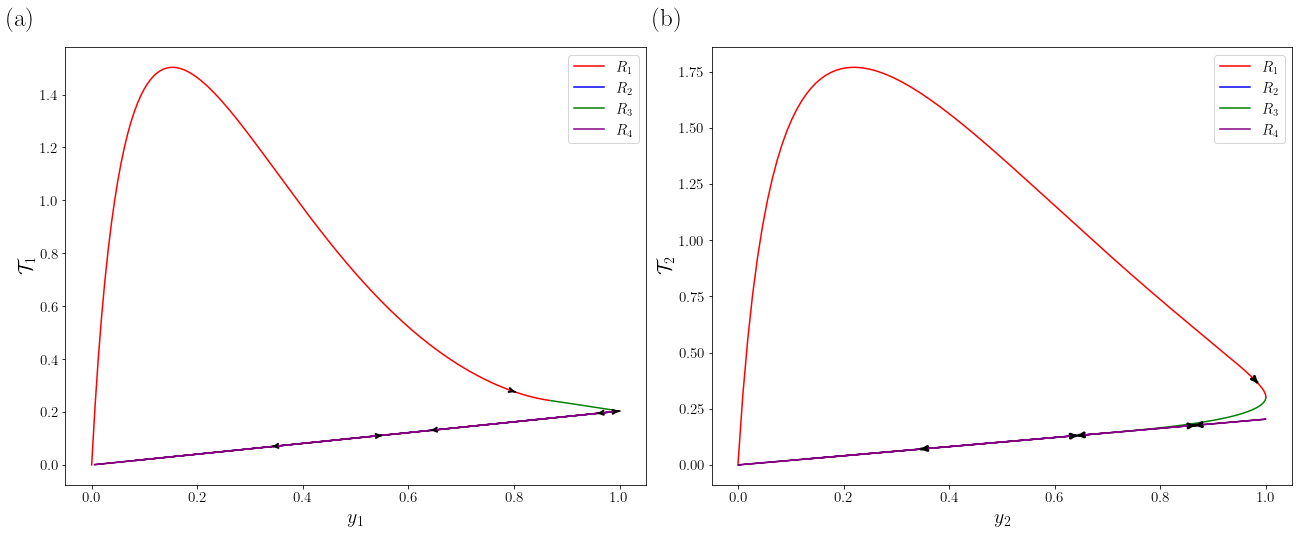}
		\caption{CASE 2: non potential-based model. a) Traction-displacement slip relation $\mc T_1 - y_1$, b) Traction-displacement slip relation $\mc T_2 - y_2$.} 
		\label{fig_Different_v_N} 
	\end{figure}
	
	\begin{figure}[htpb]
		\centering
		\includegraphics[width=0.95\textwidth]{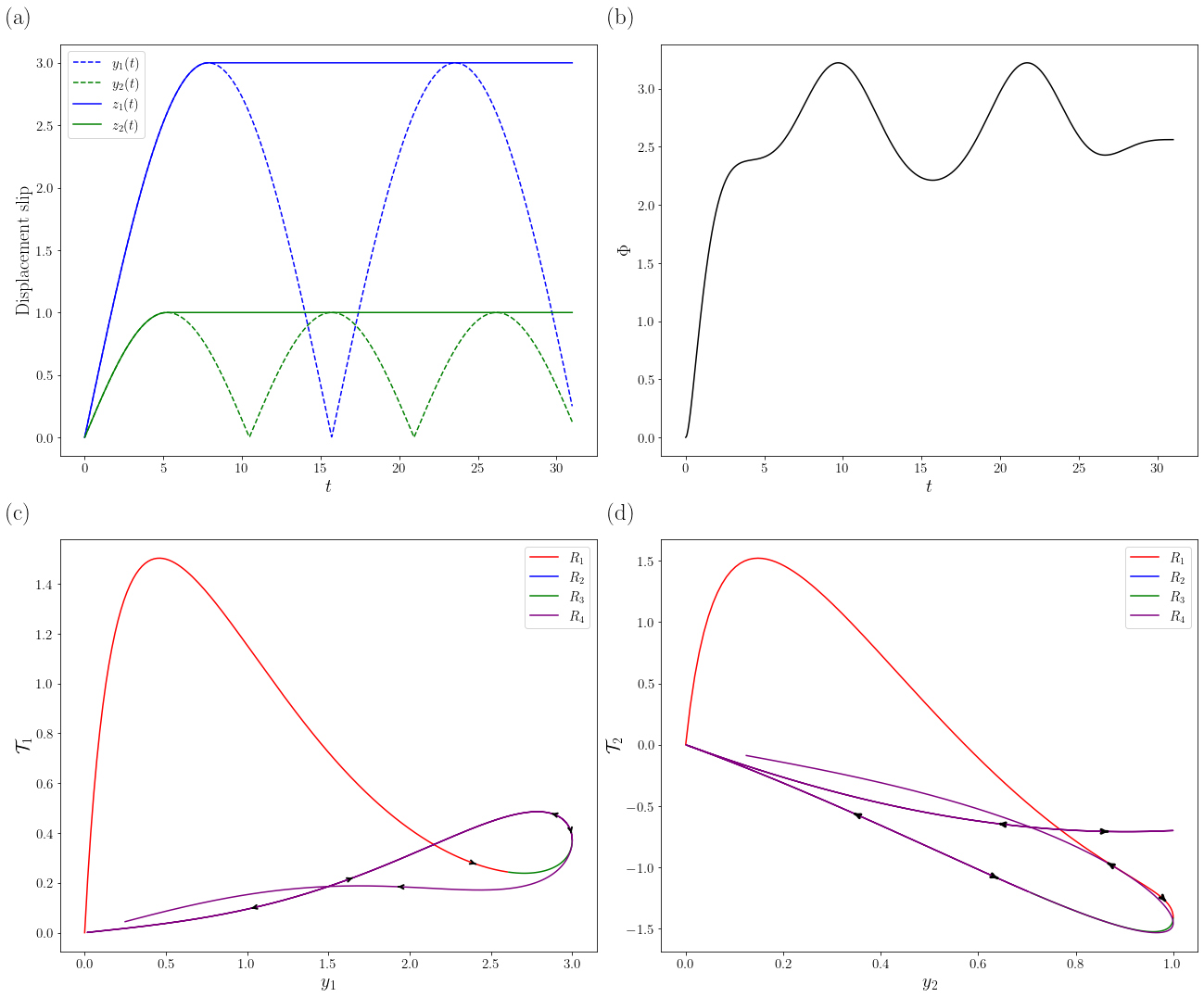}
		\caption{CASE 3: potential-based model. a) Evolution of displacement slip values $y_1, y_2$ and history variables $z_1, z_2$, b) Energy evolution, c) Traction-displacement slip relation $\mc T_1 - y_1$, d) Traction-displacement slip relation $\mc T_2 - y_2$.} 
		\label{fig_All_different} 
	\end{figure}
	
	\begin{figure}[htpb]
		\centering
		\includegraphics[width=0.95\textwidth]{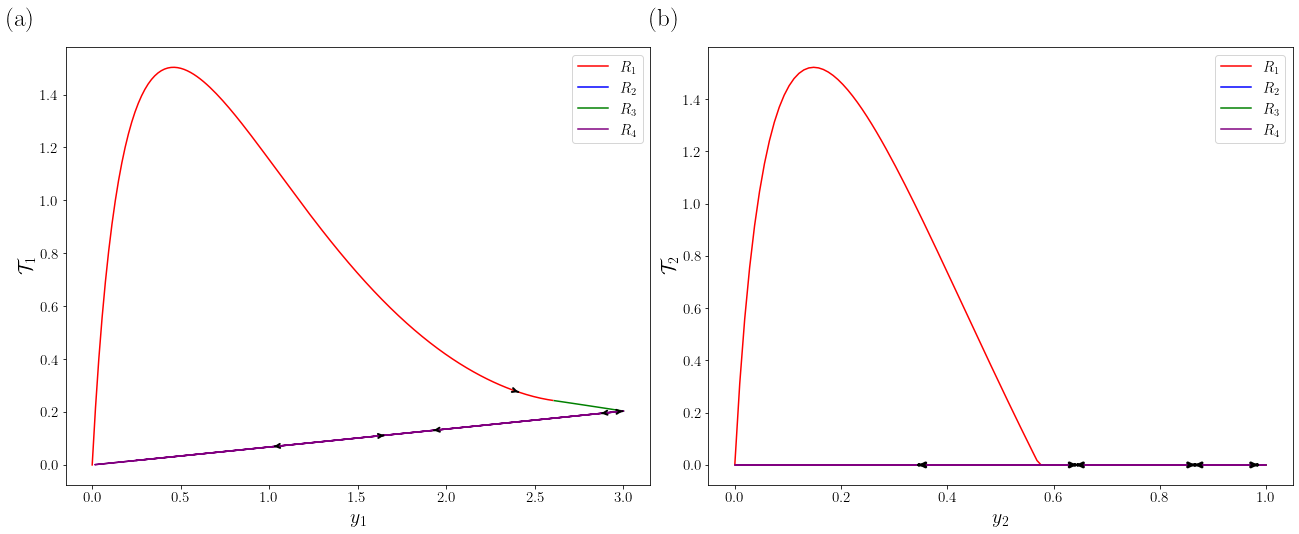}
		\caption{CASE 3: non potential-based model. a) Traction-displacement slip relation $\mc T_1 - y_1$, b) Traction-displacement slip relation $\mc T_2 - y_2$.} 
		\label{fig_All_different_N} 
	\end{figure}
	
	\begin{figure}[htpb]
		\centering
		\includegraphics[width=0.95\textwidth]{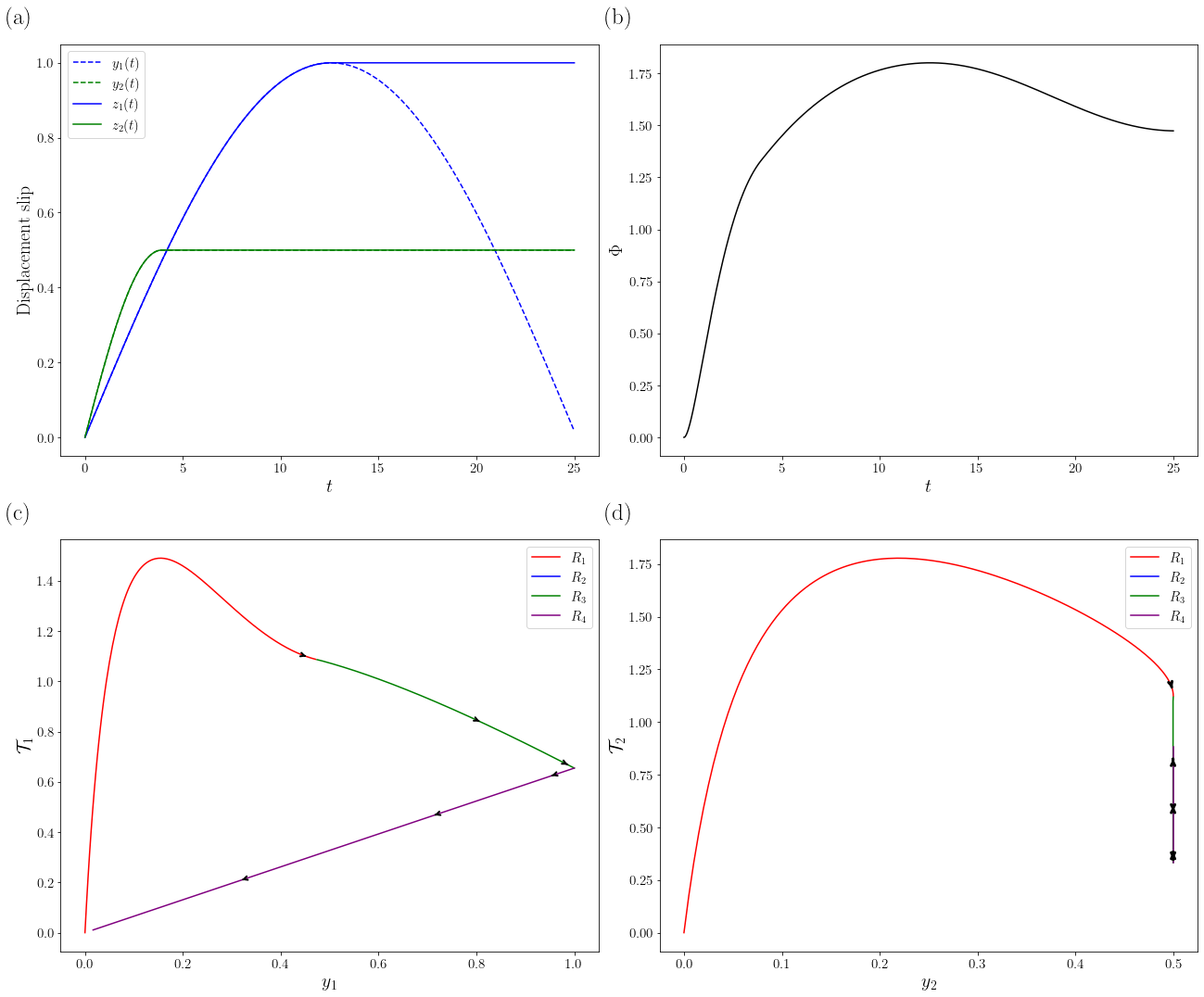}
		\caption{CASE 4: potential-based model. a) Evolution of displacement slip values $y_1, y_2$ and history variables $z_1, z_2$, b) Energy evolution, c) Traction-displacement slip relation $\mc T_1 - y_1$, d) Traction-displacement slip relation $\mc T_2 - y_2$.} 
		\label{fig_PPR_funziona} 
	\end{figure}

	\FloatBarrier
	
	\section{Proof of the existence results}\label{sec:existence}
	This last section is devoted to the proof of Theorems~\ref{thm:existence} and \ref{thm:existencenonpot}; we thus tacitly assume all their hypotheses, when needed. 
	\subsection{Energetic solutions}
	We first consider the potential-based model. The argument is similar to the one developed in \cite{BonCavFredRiva,Riv}, so we only sketch the various proofs stressing the differences which arise due to the anisotropy. We begin by performing a time-discretization algorithm. Let $\tau>0$ such that $T/\tau\in \N$, and for $k=0,\dots, T/\tau$ we define $t^k:=k\tau$. For $k=1,\dots,T/\tau$ we now consider the following recursive minimization scheme: given $(\bm u^{k-1},\gamma^{k-1})$, we set
	\begin{equation}\label{eq:discralg}
		\begin{cases}\displaystyle
			\bm u^k\in\argmin\limits_{\bm{v}\in H^1(\Omega;\R^d)^2}\mc F(t^k,\bm{v}, \gamma^{k-1}),\\
			\gamma^k:=\gamma^{k-1}\bm\vee\mf g(u_1^k-u_2^k),			
		\end{cases}
	\end{equation}
	where the initial conditions naturally are the given pair of initial data $(\bm u^0,\gamma^0)$.
	
	The existence of the minimum in \eqref{eq:discralg}$_1$ directly follows from the Direct Method of the Calculus of Variations in the weak topology of $H^1(\Omega;\R^d)^2$. Coercivity is ensured by Korn-Poincarè inequality, since the cohesive energy $\mc K$ is nonnegative; on the other hand, lower semicontinuity of the elastic energy $\mc E$ is standard, while it follows by Fatou's lemma for $\mc K$.
	
	A crucial tool in order to gain compactness for the discrete history variable $\gamma^k$ will be the following regularity result, whose proof can be found in \cite[Theorem~7.2]{GiacMart}.
	\begin{thm}\label{thm:regularity}
		Let the set $\Omega\subseteq \R^d$ be bilipschitz diffeomorphic to the open unit cube. Let $u\in H^1(\Omega;\R^d)$ be a weak solution of the equation
		\begin{equation*}
			-\div(\mathbb{C}e(u))=g,\qquad\text{in }\Omega,
		\end{equation*}
		where the tensor $\mathbb{C}$ satisfies \ref{hyp:C1}, \ref{hyp:C3} and \eqref{eq:LegendreHadamard}.
		
		If $g\in L^{\frac{dp}{d+p}}(\Omega;\R^d)$ for some $p>2$, then $\nabla u\in L^p_{\rm loc}(\Omega;\R^{d\times d})$ and for all open set $\Omega'\subset\subset\Omega$ there holds
		\begin{equation*}
			\|\nabla u\|_{L^p(\Omega')}\le C'\Big(\| g\|_{L^{\frac{dp}{d+p}}(\Omega)}+\|\nabla u\|_{L^2(\Omega)}\Big),
		\end{equation*}
		where the constant $C'>0$ depends only on $p,d,c,\omega$ and $\mathrm{dist}(\Omega',\Omega)$.
	\end{thm}
	This first lemma provides boundedness of the discrete displacements $\bm u^k$, by exploiting their minimality property \eqref{eq:discralg}$_1$
	\begin{lemma}
		There exists a constant $C>0$ independent of $\tau$ such that
		\begin{equation}\label{eq:boundH1}
			\max\limits_{k=0,\dots,T/\tau}\| \bm u^k\|_{H^{1}(\Omega)^2}\le C.
		\end{equation}
	\end{lemma}
	\begin{proof}
		For $k=0$ there is nothing to prove. So we fix $k\ge 1$ and we pick $\bm\ell(t^k):=(\ell(t^k),\ell(t^k))$ as a competitor for $\bm u^k$ in \eqref{eq:discralg}$_1$. By employing \ref{hyp:C1}, \ref{hyp:C5}, \ref{hyp:g1} and \ref{hyp:phi2} we estimate
		\begin{align*}
			\sum_{i=1}^2 \frac{c_i}{2}\| e(u_i^k)\|^2_{L^2(\Omega)}&\le \mc F(t^k,\bm u^k,\gamma^{k-1})\le \mc F(t^k,\bm \ell(t^k),\gamma^{k-1})=\mc E(\bm \ell(t^k))+\mc K(0_d,\gamma^{k-1})\\
			&\le C\| \ell(t^k)\|^2_{H^{1}(\Omega)}+\int_{\Omega}\Phi(x,0_m,\gamma^{k-1})\d x\le C(\max\limits_{t\in[0,T]}\| \ell(t)\|^2_{H^{1}(\Omega)}+1).
		\end{align*}
		We now conclude by means of \eqref{eq:externalloading} and Korn-Poincarè inequality. 
	\end{proof}
	By somehow computing the Euler-Lagrange equations of $\mc F(t^k,\cdot,\gamma^{k-1})$, see also \eqref{eq:EL}, we deduce the following uniform estimate.
	\begin{lemma}
		There exists a constant $C>0$ independent of $\tau$ such that for $i=1,2$ and for any $\varphi\in H^1_0(\Omega;\R^d)$ there holds
		\begin{equation}\label{eq:boundLinfty}
			\max\limits_{k=0,\dots,T/\tau}|\langle\div\C_i e(u^k_i),\varphi\rangle_{H^1_0(\Omega;\R^d)}|\le C\|\varphi\|_{L^1(\Omega)}.
		\end{equation}
	\end{lemma}
	\begin{proof}
		Fix $k=0,\dots,T/\tau$ and let $\bm\varphi\in H^1_0(\Omega;\R^d)^2$ and $h>0$. By taking $\bm u^k+h\bm\varphi$ as a competitor for $\bm u^k$ in \eqref{eq:discralg}$_1$ and \eqref{eq:initialstability} (we set $\gamma^{-1}:=\gamma^0$) we obtain
		\begin{align*}
			0&\le \liminf\limits_{h\to 0^+}\frac{\mc F(t^k,\bm u^k+h\bm\varphi,\gamma^{k-1})-\mc F(t^k,\bm u^k,\gamma^{k-1})}{h}\\
			&=\sum_{i=1}^2\int_\Omega \C_i e(u_i^k):e(\varphi_i)\d x+\liminf\limits_{h\to 0^+}\frac{\mc K(u_1^k-u_2^k+h(\varphi_1-\varphi_2),\gamma^{k-1})-\mc K(u_1^k-u_2^k,\gamma^{k-1})}{h}.
		\end{align*}
		Observing that by \ref{hyp:g2} and \ref{hyp:phi3} there holds
		\begin{equation*}
			|\mc K(u_1^k-u_2^k+h(\varphi_1-\varphi_2),\gamma^{k-1})-\mc K(u_1^k-u_2^k,\gamma^{k-1})|\le C h \|\varphi_1-\varphi_2\|_{L^1(\Omega)},
		\end{equation*}
		we thus deduce
		\begin{equation*}
			\sum_{i=1}^2\langle\div\C_i e(u^k_i),\varphi_i\rangle_{H^1_0(\Omega;\R^d)}\le  C\|\varphi_1-\varphi_2\|_{L^1(\Omega)}.
		\end{equation*}
		By choosing $\varphi_1=0_d$ or $\varphi_2=0_d$ we finally conclude.
	\end{proof}
	We now employ the above two lemmas, together with Theorem~\ref{thm:regularity}, in order to improve the previous uniform bounds.
	\begin{lemma}\label{lemma:b}
		For $k=0,\dots, T/\tau$ there holds
		\begin{subequations}
			\begin{align}
				\bm u^k\in W^{1,p}_{\rm loc}(\Omega;\R^d)^2\quad &\text{for all }p>2,\label{eq:uW1p}\\
				\gamma^k\in C^{0,\alpha}_{\rm loc}(\Omega;\R^m)\quad&\text{for all }\alpha\in (0,1),\label{eq:gammaHolder}
			\end{align}
		\end{subequations}
		and for every $\Omega'\subset\subset\Omega$ there exists a constant $C'>0$ independent of $\tau$ (possibly depending on $p$ and $\Omega'$) such that
		\begin{subequations}\label{eq:bounds}
			\begin{align}
				\max\limits_{k=0,\dots, T/\tau}\| \bm u^k\|_{W^{1,p}(\Omega')^2}&\le C',\label{eq:boundW1p}\\
				\max\limits_{k=0,\dots, T/\tau}\|  \gamma^k\|_{\rm{C^{0,\alpha}}(\overline\Omega')}&\le C'.\label{eq:boundHoldergamma}
			\end{align}
		\end{subequations}
	\end{lemma}
	\begin{proof}
		We fix $k=0,\dots, T/\tau$, $p>2$ and we first observe that by intersecting the set $\Omega$ with a sufficiently fine cubic grid we may write it as
		\begin{equation*}
			\Omega=\bigcup_{j=1}^N\Omega_j,\text{ where each $\Omega_j$ is bilipschitz diffeomorphic to the open unit cube in $\R^d$}.
		\end{equation*}
		Pay attention that we are not requiring the subsets to be disjoint. We start working in a single set $\Omega_j$ introduced above. For $i=1,2$, by \eqref{eq:boundLinfty} we infer that
		\begin{equation*}
			\div\C_i e(u^k_i)\in L^\infty(\Omega_j;\R^d),\quad\text{with}\quad \|\div\C_i e(u^k_i)\|_{L^\infty(\Omega_j)}\le C.
		\end{equation*}
		The regularity Theorem~\ref{thm:regularity} now yields $\nabla u^k_i\in L^q_{\rm loc}(\Omega_j;\R^{d\times d})$ for any $q>2$ with
		\begin{equation*}
			\| \nabla u^k_i\|_{L^q(\Omega'_j)}\le C'_j(\|\div\C_i e(u^k_i)\|_{L^\infty(\Omega_j)}+\|\nabla u^k_i\|_{L^2(\Omega_j)})\le C'_j,\quad\text{for any }\Omega'_j\subset\subset\Omega_j,
		\end{equation*}
		where we exploited \eqref{eq:boundH1} in the last inequality.
		
		By arguing as in \cite[Proposition~2.5]{Riv} we then deduce that $ u^k_i\in W^{1,p}_{\rm loc}(\Omega_j;\R^{d})$ with
		\begin{equation*}
			\|u^k_i\|_{W^{1,p}(\Omega'_j)}\le C'_j,\quad\text{for any }\Omega'_j\subset\subset\Omega_j.
		\end{equation*}
		This readily implies \eqref{eq:uW1p} and \eqref{eq:boundW1p}, indeed for any $\Omega'\subset\subset \Omega$ one can easily find sets $\Omega_j'\subset\subset \Omega_j$ such that $\Omega'=\bigcup_{j=1}^N\Omega'_j$.
		
		We now fix $k=0,\dots, T/\tau$ and $\alpha\in (0,1)$. If $k=0$ there is nothing to prove since $\gamma^0\in C^{0,1}_{\rm loc}(\Omega;\R^m)$. If $k\ge 1$ we consider $p>\frac{d}{1-\alpha}$, so that by Sobolev embedding one actually has $\bm u^n\in  C^{0,\alpha}_{\rm loc}(\Omega;\R^d)^2$ for any $n=1\dots, k$ with 
		\begin{equation*}
			\max\limits_{n=1,\dots,k}\|\bm u^n\|_{C^{0,\alpha}(\overline{\Omega'})^2}
			\le C'\max\limits_{n=1,\dots,k}\|\bm u^n\|_{W^{1,p}(\Omega'')^2}\le C',\quad\text{for any }\Omega'\subset\subset\Omega,
		\end{equation*}
		where $\Omega''$ is an open set with Lipschitz boundary such that $\Omega'\subset\subset\Omega''\subset\subset\Omega$.
		
		In particular, since $\mf g$ is continuous, for any $l=1,\dots, m$ one deduces
		\begin{equation*}
			\max\limits_{n=1,\dots,k}\|\mf g_l(u_1^n-u_2^n)\|_{C^{0}(\overline{\Omega'})}\le C'.
		\end{equation*}
		Moreover the discrete history variable $\gamma^k_l=\max\limits_{n=1,\dots,k}\mf g_l(u_1^n-u_2^n)\vee \gamma^0_l$ is continuous in $\Omega$ since it is a finite maximum of continuous functions, and for any $\Omega'\subset\subset\Omega$ there holds
		\begin{equation*}
			\|\gamma^k_l\|_{C^0(\overline{\Omega'})}=\max\limits_{n=1,\dots,k}\|\mf g_l(u_1^n-u_2^n)\|_{C^0(\overline{\Omega'})}\vee \|\gamma^0_l\|_{C^0(\overline{\Omega'})}\le C'.
		\end{equation*}
		
		In order to show the validity of \eqref{eq:gammaHolder} and \eqref{eq:boundHoldergamma} we just need to control the H\"older seminorms $[\gamma^k_l]_{\alpha,\overline{\Omega'}}$ for an arbitrary set $\Omega'\subset\subset\Omega$. Let us fix two different points $x,y\in \overline{\Omega'}$; then there are two possibilities: either $\gamma^k_l(x)=\gamma^0_l(x)$ or there exists $\bar n\in \{1,\dots, k\}$ such that $\gamma^k_l(x)=\mf g_l(u_1^{\bar n}(x)-u_2^{\bar n}(x))$. In the first case we can estimate
		\begin{equation*}
			\gamma^k_l(x)\le \gamma^0_l(y)+|\gamma^0_l(x)-\gamma^0_l(y)|\le \gamma^k_l(y)+C'|x-y|_d^\alpha.
		\end{equation*}
		In the second case, instead, by using \ref{hyp:g2} we have
		\begin{align*}
			\gamma^k_l(x)&\le \mf g_l(u_1^{\bar n}(y)-u_2^{\bar n}(y))+|\mf g_l(u_1^{\bar n}(x)-u_2^{\bar n}(x))-\mf g_l(u_1^{\bar n}(y)-u_2^{\bar n}(y))|\\
			&\le \gamma^k_l(y)+ C\sum_{i=1}^2|u_i^{\bar n}(x)-u_i^{\bar n}(y)|_d\le \gamma^k_l(y)+C'|x-y|^\alpha_d.
		\end{align*}
		By the arbitrariness of $x$ and $y$ we finally conclude.
	\end{proof}
	
	We now introduce the piecewise constant interpolants $(\bm u^\tau,\gamma^\tau)$ of the discrete displacements and history variable defined as
	\begin{equation}\label{eq:piecewise}
		\begin{cases}
			\bm u^\tau(t):=\bm u^k,& \gamma^\tau(t):=\gamma^k,\qquad\text{if }t\in[t^k,t^{k+1}),\\
			\bm u^\tau(T):=\bm u^{ T/\tau},& \gamma^\tau(T):=\gamma^{ T/\tau}.
		\end{cases}
	\end{equation}
	For conveniency, we also set $w^\tau$ as 
	\begin{equation*}
		\begin{cases}
			w^\tau(t):=w(t^k),\qquad\text{if }t\in[t^k,t^{k+1}),\\
			w^\tau(T):=w(T).
		\end{cases}
	\end{equation*}
	For a given $t\in [0,T]$ we finally define
	\begin{equation*}
		t^\tau:=\max\{t^k:\,t^k\le t\}.
	\end{equation*}
	We first show that such interpolants satisfy a discrete energy inequality.
	\begin{prop}
		For every $t\in [0,T]$ and $\tau>0$ the following discrete energy inequality holds true:
		\begin{equation}\label{eq:discrineq}
			\mc F(t^\tau,\bm u^\tau(t),\gamma^\tau(t))\le \mc F(0,\bm u^0,\gamma^0)+\mc W^\tau(t)+R^\tau,
		\end{equation}
		where $\displaystyle\mc W^\tau(t)=\int_{0}^t\int_{\Omega}\sum_{i=1}^2\mathbb{C}_ie(u^\tau_i(s)):e(\dot{\ell}(s))\d x\d s$, while $R^\tau\ge 0$ is an infinitesimal remainder.
	\end{prop}
	\begin{proof}
		The proof follows arguing as in \cite[Proposition~3.1]{Riv} by observing that for any $k=1,\dots,T/\tau$ assumption \ref{hyp:phi4} yields
		\begin{equation}\label{eq:important}
			\mc K(u_1^k-u_2^k,\gamma^{k})=\mc K(u_1^k-u_2^k,\gamma^{k-1}\bm\vee\mf g(u_1^k-u_2^k))=\mc K(u_1^k-u_2^k,\gamma^{k-1}).
		\end{equation}
	\end{proof}
	In view of the bounds obtained in Lemma~\ref{lemma:b}, we now deduce the following compactness result.
	\begin{prop}\label{prop:convsubs}
		There exists a subsequence $\tau_j\searrow0$ and for all $t\in [0,T]$ there exist a further subsequence $\tau_j(t)$ (possibly depending on time), and functions $\bm u(t)\in (H^1_{D,\ell(t)}(\Omega;\R^d)\cap C^{0,\alpha}_{\rm loc}({\Omega};\R^d))^2$ and $\gamma(t)\in C^{0,\alpha}_{\rm loc}({\Omega};\R^m)$ for any $\alpha\in(0,1)$ such that for all $t\in [0,T]$ there hold:
		\begin{equation}\label{eq:convtau}
			\begin{aligned}
				&	\bm u^{\tau_j(t)}(t)\xrightharpoonup[j\to+\infty]{H^1(\Omega;\R^d)^2} \bm u(t),\qquad\text{and}\qquad	\bm u^{\tau_j(t)}(t)\xrightarrow[j\to+\infty]{} \bm u(t) \text{ locally uniformly in $\Omega$},\\
				& \gamma^{\tau_j}(t) \xrightarrow[j\to+\infty]{} \gamma(t) \text{ locally uniformly in $\Omega$}.
			\end{aligned}
		\end{equation}
		In particular one has $(\bm u(0),\gamma(0))=(\bm u^0,\gamma^0)$.\\
		Moreover for all $l=1,\dots,m$ the function $\gamma_l$ is nondecreasing in time, and
		\begin{equation}\label{eq:gammasup}
			\gamma_l(t,x)\ge \sup\limits_{s\in[0,t]}\mf g_l(u_1(s,x)-u_2(s,x)), \quad\text{for every $(t,x)\in [0,T]\times\Omega$.}
		\end{equation}
		Finally there holds $u\in B([0,T];H^1(\Omega;\R^d)^2)$ and for all $\Omega'\subset\subset\Omega$ and any $\alpha\in(0,1)$ there also hold $u\in B([0,T];C^{0,\alpha}(\overline{\Omega'};\R^d))^2)$ and $\gamma\in B([0,T];C^{0,\alpha}(\overline{\Omega'};\R^m))$.
	\end{prop}
	\begin{proof}
		The compactness result \eqref{eq:convtau} can be proved exactly as in \cite[Proposition~3.3]{Riv} once we have at our disposal the uniform bounds \eqref{eq:bounds}. In the same way, inequality \eqref{eq:gammasup} follows arguing as \cite[Proposition~3.3]{Riv} since $\mf g_l$ is a continuous function. All the other properties are simple byproducts of the convergences \eqref{eq:convtau} and the uniform bounds \eqref{eq:boundH1} and \eqref{eq:bounds}.
	\end{proof}
	We finally conclude by showing that the just obtained limit functions are an energetic solution to the potential-based cohesive interface model.
	\begin{prop}
		For every $t\in [0,T]$ the limit pair $(\bm u,\gamma)$ obtained in Proposition~\ref{prop:convsubs} satisfies the global stability condition \ref{GS} and the energy balance \ref{EB} of Definition~\ref{def:genensol}.
	\end{prop}
	\begin{proof}
		The first condition in \ref{GS} is implied by \eqref{eq:gammasup}, while the global minimality property follows as in \cite[Proposition~3.5]{Riv} recalling assumptions \ref{hyp:phi4}, \ref{hyp:phi5} and property \eqref{eq:important}.
		
		As a consequence of \ref{GS}, arguing as in \cite[Proposition~3.10]{BonCavFredRiva} one can prove the lower energy inequality
		\begin{equation*}
			\mc F(t,\bm u(t),\gamma(t))\ge \mc F(0,\bm u^0,\gamma^0)+\mc W(t),
		\end{equation*}
		by exploiting assumption \ref{hyp:phi5}. The opposite inequality, which finally yields \ref{EB}, instead follows by sending $\tau_j\to 0$ in the discrete energy inequality \eqref{eq:discrineq}, see \cite[Proposition~3.4]{Riv} for more details.
	\end{proof}
	
	\subsection{Equilibrium solutions}
	We now focus on Theorem~\ref{thm:existencenonpot}. We discretize the time interval $[0,T]$ as in the previous section, but here we need to consider a different recursive scheme. Given a pair $(\bm u^{k-1}, \gamma^{k-1})$ we first select $\bm u^k\in (H^1_{D,\ell(t^k)}(\Omega;\R^d))^2$ as a solution to \eqref{eq:system} with history variable at the previous step $\gamma^{k-1},$ namely satisfying
	\begin{subequations}\label{eq:sist2}
		\begin{equation}\label{eq:weaksoldiscr}
			\sum_{i=1}^2\int_\Omega \mathbb C_ie(u_i^k):e(\varphi_i) \d x=-\int_\Omega\mathcal T(x,\mathfrak g(u_1^k-u_2^k),\gamma^{k-1})\cdot (\eta_1^k\varphi_1-\eta_2^k\varphi_2) \d x,
		\end{equation}
		for all $\bm\varphi \in (H^1_{D,0_d}(\Omega;\R^d))^2$ and for some 
		\begin{equation}
			\bm \eta^k\in (L^\infty(\Omega;\R^{m\times d}))^2 \text{ such that } \eta^k_i(x)\in D\mathfrak g(u_1^k(x)-u_2^k(x))\text{ for a.e. } x\in\Omega.
		\end{equation}
	\end{subequations} 
	Then we define $\gamma^k$ as in \eqref{eq:discralg}$_2$. 
	
	The existence of a solution $\bm u^k$ to \eqref{eq:sist2} is granted by the well-known Kakutani's fixed point theorem (or better, by its infinite-dimensional generalization \cite{Fan,Glick}).
	
	\begin{thm}\label{Kakutani}
		Let $S$ be a nonempty convex compact subset of a Hausdorff locally convex topological vector space, and consider a set-valued function $\mc R\colon S\to 2^S$ satisfying:
		\begin{itemize}
			\item $\mc R(s)$ is nonempty and convex for all $s\in S$;
			\item the graph of $\mc R$ is closed, i.e. if $s_n\to s$, $u_n\to u$, $s_n\in S$ and $u_n\in \mc R(s_n)$, then $u\in \mc R(s)$.
		\end{itemize}
		Then, there exists a fixed point $s\in \mc R(s)$.
	\end{thm}
	\begin{cor}\label{cor:fixedpoint}
		Under the assumptions of Theorem~\ref{thm:existencenonpot}, system \eqref{eq:sist2} admits a solution.
	\end{cor}
	\begin{proof}
		It is enough to show that the following set-valued map admits a fixed point. We define $S:=\left\{\bm s\in  (H^1_{D,\ell(t^k)}(\Omega;\R^d))^2:\, \|s_i\|_{H^1(\Omega)}\le R\text{ for }i=1,2\right\}$, with $R>0$ to be chosen, and we consider $\mc R^k\colon S\to 2^S$ which maps a function $\bm s\in S$ to the set of functions $\bm u\in (H^1_{D,\ell(t^k)}(\Omega;\R^d))^2$ solving
		\begin{equation}\label{eq:eqs}
			\sum_{i=1}^2\int_\Omega \mathbb C_ie(u_i):e(\varphi_i) \d x=-\int_\Omega\mathcal T(x,\mathfrak g(s_1-s_2),\gamma^{k-1})\cdot (\eta_1\varphi_1-\eta_2\varphi_2) \d x,
		\end{equation}
		for all $\bm\varphi \in (H^1_{D,0_d}(\Omega;\R^d))^2$, as $\bm \eta\in (L^\infty(\Omega;\R^{m\times d}))^2$ such that $\eta_i(x)\in D\mathfrak g(s_1(x)-s_2(x))$ for a.e. $x\in\Omega$ varies.
		
		The set $S$ is clearly nonempty and convex; moreover, it is compact if endowed with the weak topology of $(H^1(\Omega;\R^d))^2$ (which is also metrizable on $S$, since it is bounded). Let us first show that $\mc R^k$ is valued in $S$, up to choosing $R$ large enough. Given $\bm u\in \mc R^k(\bm s)$, by taking as a test function $\bm \varphi=\bm u-\bm\ell(t^k)$ we infer
		\begin{align}\label{eq:H1bound2}
			\sum_{i=1}^2c_i\|e(u_i)\|^2_{L^2(\Omega)}&\le \sum_{i=1}^2\int_\Omega \mathbb C_ie(u_i):e(u_i-\ell(t^k))  \d x+\int_\Omega\mathbb C_ie(u_i):e(\ell(t^k)) \d x\nonumber\\
			&\le C\|\mc T\|_{L^\infty}\|\mathfrak g\|_{C^{0,1}(\R^d)}\sum_{i=1}^2\|u_i-\ell(t^k)\|_{L^2(\Omega)}+C\sum_{i=1}^2\|e(u_i)\|_{L^2(\Omega)}\\
			&\le C \left(\sum_{i=1}^2\|e(u_i)\|_{L^2(\Omega)}+1\right),\nonumber
		\end{align}
		where we exploited assumptions \ref{hyp:T1}, \ref{hyp:g2} together with Korn-Poincar\'e inequality. Notice that the constant $C$ does not depend on $\bm s$, thus the above chain of inequalities implies that $\|e(u_i)\|_{L^2(\Omega)}$ is uniformly bounded, and so again by Korn-Poincar\'e inequality one infers that $\bm u\in S$ if $R$ is large.
		
		Recalling that by \ref{hyp:g3} the set $D\mathfrak g(\delta)$ is convex for all $\delta\in \R^d$, we also easily deduce that $\mc R^k(\bm s)$ is convex as well (and clearly nonempty) for all $\bm s\in S$.
		
		In order to apply Theorem~\ref{Kakutani} and conclude, we just need to show that the graph of $\mc R^k$ is closed with respect to the weak topology of $(H^1(\Omega;\R^d))^2$. To this aim, consider sequences $\bm s^n\rightharpoonup \bm s$, $\bm u^n\rightharpoonup \bm u$ weakly in $(H^1(\Omega;\R^d))^2$, such that $\bm u^n\in \mc R^k(\bm s^n)$. In particular, let $\bm \eta^n\in (L^\infty(\Omega;\R^{m\times d}))^2$ such that $\eta^n_i(x)\in D\mathfrak g(s^n_1(x)-s^n_2(x))$ for a.e. $x\in\Omega$ satisfying \eqref{eq:eqs}  given by the definition of $\mc R^k$. Since $\mathfrak g$ is Lipschitz, then $\bm \eta^n$ is uniformly bounded in $L^\infty$, so without loss of generality we may assume that $\bm \eta^n\xrightharpoonup[]{*}\bm \eta$ weakly$^*$ in $(L^\infty(\Omega;\R^{m\times d}))^2$. Hence, by means of \ref{hyp:g4}, we deduce that $\eta_i(x)\in D\mathfrak g(s_1(x)-s_2(x))$ for a.e. $x\in\Omega$. Moreover, equation \eqref{eq:eqs} passes to the limit: indeed, on the left-hand side we may use weak convergence in $H^1$, while on the right-hand side we exploit the fact that 
		\begin{equation*}
			\mathcal T(\cdot,\mathfrak g(s^n_1-s^n_2),\gamma^{k-1})\eta^n_i\rightharpoonup\mathcal T(\cdot,\mathfrak g(s_1-s_2),\gamma^{k-1})\eta_i,\quad\text{weakly in }L^2(\Omega;\R^d),
		\end{equation*}
		since it is the product of the weak convergent sequence $\eta_i^n$ and of the strong convergent one $\mathcal T(\cdot,\mathfrak g(s^n_1-s^n_2),\gamma^{k-1})$ (this latter property follows by Dominated Convergence Theorem in view of \ref{hyp:T1} and \ref{hyp:g2}), and both of them are bounded in $L^\infty$.
		
		We have thus proved that $\bm u\in \mc R^k(\bm s)$, namely the graph of $\mc R^k$ is closed, and we conclude.
	\end{proof}
	
	From now on, the strategy is similar to the one presented in the previous section. We first show uniform bounds for the sequence of pairs $(\bm u^{k}, \gamma^{k})$, which yield convergence of the corresponding piecewise constant interpolants. We then prove that the obtained limit are indeed an equilibrium solution to the non potential-based cohesive interface model. 
	
	\begin{prop}
		There exists a constant $C>0$ independent of $\tau$ such that the bounds \eqref{eq:boundH1} and \eqref{eq:bounds} hold true.
	\end{prop}
	\begin{proof}
		The $H^1$ bound \eqref{eq:boundH1} can be obtained as in \eqref{eq:H1bound2} by taking as a test function $\bm \varphi=\bm u^k-\bm\ell(t^k)$. Instead, the bounds \eqref{eq:bounds} follow by arguing exactly as in the proof of Lemma~\ref{lemma:b}: the validity of \eqref{eq:boundLinfty} in the current framework can be indeed directly checked from \eqref{eq:weaksoldiscr} by means of \ref{hyp:T1} and \ref{hyp:g2}.
	\end{proof}
	
	We now consider the piecewise constant interpolants $(\bm u^\tau,\gamma^\tau)$ defined as in \eqref{eq:piecewise}. Due to the just obtained uniform bounds, we immediately deduce that the same results of Proposition~\ref{prop:convsubs} hold true also in the current setting. We thus conclude the proof of Theorem~\ref{thm:existencenonpot} if we show that the limit pair $(\bm u(t),\gamma(t))$ solves the equilibrium equation \eqref{eq:weaksol} for all times.
	
	To this aim, we first introduce the piecewise constant interpolant $\bm \eta^\tau$ defined as in \eqref{eq:piecewise} where the discrete values $\bm\eta^k$ are given by \eqref{eq:sist2}. Notice that, by \ref{hyp:g3}, without loss of generality we may assume that
	\begin{equation*}
		\bm \eta^{\tau_j(t)}(t)\xrightharpoonup[j\to+\infty]{L^\infty(\Omega;\R^{m\times d})^2_*} \bm \eta(t),\quad\text{ for all }t\in [0,T].
	\end{equation*}
	Moreover, by using \ref{hyp:g4}, we also infer that $\eta_i(t,x)\in D\mathfrak g(u_1(t,x)-u_2(t,x))$ for a.e. $x\in\Omega$. Then, fix $t\in [0,T]$ and $\bm\varphi \in (H^1_{D,0_d}(\Omega;\R^d))^2$. By \eqref{eq:weaksoldiscr} and exploiting \ref{hyp:T2} together with the definition of $\gamma^k$ we observe that
	\begin{align*}
		&\sum_{i=1}^2\int_\Omega \mathbb C_ie(u_i^{\tau_j(t)}(t)):e(\varphi_i) \d x\\
		=&-\int_\Omega\mathcal T(x,\mathfrak g(u_1^{\tau_j(t)}(t)-u_2^{\tau_j(t)}(t)),\gamma^{\tau_j(t)}(t-{\tau_j(t)}))\cdot (\eta_1^{\tau_j(t)}(t)\varphi_1-\eta_2^{\tau_j(t)}(t)\varphi_2) \d x,\\
		=&-\int_\Omega\mathcal T(x,\mathfrak g(u_1^{\tau_j(t)}(t)-u_2^{\tau_j(t)}(t)),\gamma^{\tau_j(t)}(t))\cdot (\eta_1^{\tau_j(t)}(t)\varphi_1-\eta_2^{\tau_j(t)}(t)\varphi_2) \d x.
	\end{align*}
	By means of \eqref{eq:convtau}, and arguing as in the last part of Corollary~\ref{cor:fixedpoint}, we can pass to the limit the first and the last line above by weak convergence in $H^1$ and $L^2$, respectively, exploiting \ref{hyp:T1} and \ref{hyp:g2}. Thus, equation \eqref{eq:weaksol} is satisfied and Theorem~\ref{thm:existencenonpot} is proved.

	\bigskip
	
	\noindent\textbf{Acknowledgements.} The authors wish to thank F. Iurlano for fruitful discussions on the topic. F. R. is a member of GNAMPA--INdAM.

	{\small
		
		\vspace{15pt} (Francesco Freddi) Universit\`{a} degli Studi di Parma, Dipartimento di Ingegneria e Architettura, \par \textsc{Parco Area delle Scienze, 181/A, 43124, Parma, Italy}
		\par
		\textit{e-mail address}: \textsf{francesco.freddi@unipr.it}
		\par
		\textit{Orcid}: \textsf{https://orcid.org/0000-0003-0601-6022}
		\par
		
		\vspace{10pt} (Filippo Riva) Universit\`{a} Commerciale Luigi Bocconi, Dipartimento di Scienze delle Decisioni, \par
		\textsc{via Roentgen 1, 20136 Milano, Italy}
		\par
		\textit{e-mail address}: \textsf{filippo.riva@unibocconi.it}
		\par
		\textit{Orcid}: \textsf{https://orcid.org/0000-0002-7855-1262}
		\par
		
	}
	
\end{document}